\documentclass[12pt,letterpaper]{article}

\pdfoutput=1

\usepackage{graphics,epsfig,graphicx,color, pgfplotstable, booktabs}
\graphicspath{{../Figures/}}
\usepackage[caption = false]{subfig}
\usepackage{float}
\usepackage{capt-of}

\usepackage{amssymb,latexsym}

\usepackage{setspace}

\usepackage{amsmath}

\usepackage{mathrsfs,amsfonts}

\usepackage{amsthm,amsxtra}

\usepackage[final]{showkeys}

\usepackage{stmaryrd}

\usepackage{algorithm}
\usepackage{algorithmicx}
\usepackage{algpseudocode}
\makeatletter
\renewcommand{\ALG@name}{Subroutine}
\makeatother

\usepackage[openbib]{currvita}

\newif\ifPDF
\ifx\pdfoutput\undefined
	\PDFfalse
\else
	\ifnum\pdfoutput > 0
		\PDFtrue
	\else
		\PDFfalse
	\fi
\fi

\ifPDF
	\usepackage{pdftricks}
	\begin{psinputs}
		\usepackage{pstricks}
		\usepackage{pstcol}
		\usepackage{pst-plot}
		\usepackage{pst-tree}
		\usepackage{pst-eps}
		\usepackage{multido}
		\usepackage{pst-node}
		\usepackage{pst-eps}
	\end{psinputs}
\else
	\usepackage{pstricks}
\fi

\ifPDF
	\usepackage[debug,pdftex,colorlinks=true, 
	linkcolor=blue, bookmarksopen=false,
	plainpages=false,pdfpagelabels]{hyperref}
\else
	\usepackage[dvips]{hyperref}
\fi

\usepackage{fancyhdr}

\usepackage{comment}

\usepackage[toc,page]{appendix}

\usepackage{cancel}

\newtheorem{theorem}{Theorem}[section]
\newtheorem{lemma}[theorem]{Lemma}

\newtheorem{corollary}[theorem]{Corollary}

\newcommand{\dint}{\displaystyle\int}

\newcommand{\eps}{\varepsilon}

\newcommand{\bbR}{\mathbb R} \newcommand{\bbS}{\mathbb S}

\newcommand{\bzero}{{\mathbf 0}}

 \newcommand{\bn}{\mathbf n}

\newcommand{\bu}{\mathbf u} \newcommand{\bv}{\mathbf v} 
 \newcommand{\bx}{\mathbf x}

\newcommand{\cA}{\mathcal A} \newcommand{\cB}{\mathcal B}
\newcommand{\cC}{\mathcal C}  
\newcommand{\cE}{\mathcal E} 
 \newcommand{\cH}{\mathcal H}

\newcommand{\cM}{\mathcal M} 
\newcommand{\cO}{\mathcal O}  
 \newcommand{\cR}{\mathcal R}

\newcommand{\wt}{\widetilde}
\newcommand{\wh}{\widehat}
\newcommand{\doubleline}[1]{{\overline{\overline #1}}}

\newenvironment{keywords}
{\noindent{\bf Key words.}\small}{\par\vspace{1ex}}
\newenvironment{AMS}
{\noindent{\bf AMS subject classifications 2010.}\small}{\par}

\title{Image reconstruction in quantitative photoacoustic tomography with the simplified $P_2$ approximation}

\author{
	Christina Frederick\thanks{
		Department of Mathematical Sciences, 
		New Jersey Institute of Technology,
		University Heights, New Jersey 07102; christin@njit.edu
	}
	\and
	Kui Ren\thanks{
		Department of Mathematics and Institute for Computational Engineering and Sciences,
		The University of Texas,
		Austin, TX 78712; ren@math.utexas.edu
	}
	\and
	Sarah Vall\'{e}lian\thanks{
		Department of Mathematics, 
		North Carolina State University,
		Raleigh, NC 27695; scvallel@ncsu.edu
	}
}

\begin{document}

\maketitle



\begin{abstract}
Photoacoustic tomography (PAT) is a hybrid imaging modality that intends to construct high-resolution images of optical properties of heterogeneous media from measured acoustic data generated by the photoacoustic effect. To date, most of the model-based quantitative image reconstructions in PAT are performed with either the radiative transport equation or its classical diffusion approximation as the model of light propagation. In this work, we study quantitative image reconstructions in PAT using the simplified $P_2$ equations as the light propagation model. We provide numerical evidences on the feasibility of this approach and derive some stability results as theoretical justifications.
\end{abstract}






\section{Introduction}
\label{SEC:intro}

Photoacoustic tomography (PAT) is a hybrid imaging method that couples ultrasound imaging and optical tomography via the photoacoustic effect, enabling high-resolution imaging of optical contrasts of heterogeneous media. In a typical way to induce the photoacoustic effect, a short pulse of near infra-red (NIR) light is sent into an optically heterogeneous medium, such as a piece of biological tissue, which we denote as $\Omega\in\bbR^3$. In the light propagation process, a portion of the photons are absorbed by the medium. The absorbed energy causes the medium to heat up slightly, and then cool down after the rest of the photons exit. The heating and cooling of the medium forces the medium to expand and then contract. This expansion and contraction generates a pressure field inside the medium which then propagates outwards in the form of ultrasound. The objective of PAT is to measure the ultrasound signals on the surface of the medium and use the measured data to recover information on the interior optical properties of the underlying medium. Interested readers are referred to~\cite{Beard-IF11,CoLaBe-SPIE09,Kuchment-MLLE12,KuKu-HMMI10,LiWa-PMB09,PaSc-IP07,Scherzer-Book10,Wang-DM04,Wang-IEEE08} for overviews of the physical principles as well as the practical applications of PAT.

The radiation of the photons inside the medium is described accurately by the phase-space radiative transport equation (RTE)~\cite{Arridge-IP99,Ren-CiCP10}. Let us denote by $u(\bx,\bv)$ the density of photons at location $\bx\in\Omega$, traveling in direction $\bv\in\bbS^2$ ($\bbS^{2}$ being the unit sphere in $\bbR^3$), \emph{integrated over the period of the pulse}. Then $u(\bx, \bv)$ solves:
\begin{equation}\label{EQ:RTE}
	\begin{array}{rcll}
		\bv\cdot\nabla u(\bx,\bv) + 
		\sigma_a(\bx)u(\bx, \bv)
		&=& \sigma_s(\bx)K(u)(\bx,\bv)
		&\mbox{in}\ X\\
		u(\bx,\bv) &=& f(\bx) &\mbox{on}\ \Gamma_{-}
	\end{array}
\end{equation}
Here $X=\Omega\times \bbS^{2}$ is the phase space of photon propagation, with $\Gamma_-=\{(\bx,\bv): (\bx,\bv)\in\partial\Omega\times \bbS^{2}\ \text{s.t.}\ -\bn(\bx)\cdot \bv >0\}$ denoting the incoming boundary of $X$ ($\bn(\bx)$ being the outer normal vector at $\bx\in\partial\Omega$). The positive functions $\sigma_a(\bx)$ and $\sigma_s(\bx)$ are respectively the absorption and scattering coefficients of the medium. The function $f(\bx)$ denotes the incoming illumination photon source, again integrated over the period of the pulse. To simplify the presentation, we have chosen the illumination source $f$ to be isotropic, i.e. independent of $\bv$. This is by no means technically necessary; see more discussions in the next section.

The scattering operator $K$ is defined as 
\[
K(u)(\bx, \bv)=
\dint_{\bbS^{2}}\Theta(\bv, \bv')u(\bx, \bv') \, d\bv' - u(\bx, \bv)
\]
where the kernel $\Theta(\bv, \bv')$ describes how photons traveling in direction $\bv'$ are scattered into direction $\bv$, and also satisfies the normalization condition $\int_{\bbS^{2}}\Theta(\bv, \bv') \, d \bv'=\int_{\bbS^{2}}\Theta(\bv', \bv) \, d \bv'=1,\  \forall\ \bv\in \bbS^{2}$. In practical applications in biomedical optics, $\Theta$ is often taken to be the Henyey-Greenstein phase function, which depends only on the product $\bv\cdot\bv'$. That is, $\Theta=\Theta_{HG}(\bv\cdot\bv')$~\cite{Arridge-IP99,ReBaHi-SIAM06}:
\[
\Theta_{HG}(\bv\cdot\bv')=\dfrac{1}{4\pi}\dfrac{1-g^2}{(1+g^2-2g\bv\cdot\bv')^{3/2}},
\]
where $g$ is the scattering anisotropy factor of the medium.

The pressure field generated at a position $\bx\in\Omega$, due to the photoacoustic effect, is the product of the Gr\"uneisen coefficient, denoted by $\Xi$, and the absorbed energy density at $\bx$. That is,
\begin{equation}\label{EQ:Int Data RTE}
	H(\bx)=H^{rte}(\bx) \equiv \Xi(\bx) \sigma_a(\bx)\int_{\bbS^{2}}u(\bx,\bv)\, d \bv .
\end{equation}
The Gr\"uneisen coefficient $\Xi$ is a parameter that measures the local photoacoustic efficiency of the medium.

This initial pressure field then evolves following the acoustic wave equation in the form of ultrasound. The equation for the evolution reads~\cite{Beard-IF11,Kuchment-MLLE12,LiWa-PMB09}:
\begin{equation}\label{EQ:Wave}
	\begin{array}{rcll}
		\dfrac{1}{c^2(\bx)}\dfrac{\partial^2 p}{\partial t^2} -\Delta p &=& 0, 
		&\text{in}\ \bbR_+ \times \Omega \\
		p(t,\bx)= H(\bx),\ \  
		\dfrac{\partial p}{\partial t}(t,\bx) &=& 0,&\mbox{in}\ \{0\}\times\Omega\\
		\bn\cdot \nabla p(t, \bx) &=& 0, & \mbox{on}\ \bbR_+\times\partial\Omega
	\end{array}
\end{equation}
where $p(t,\bx)$ is the pressure field, and $c(\bx)$ is the ultrasound wave speed. Note that instead of posing the wave equation in $\bbR^3$ as often done in the literature, we write it in a bounded domain with the Neumann boundary condition that mimic the effect of a reflective medium surface (coming, for instance, from the device that holds a piece of tissue to be imaged)~\cite{CoArBe-IP07}. This is, again, not technically essential, but only done to avoid using absorbing boundary conditions in the numerical simulations in Section~\ref{SEC:Num}.

The data measured in PAT are the ultrasound signals on the surface of the medium for a long enough time period $T$, that is, $p(t, \bx)_{|(0, T)\times\partial\Omega}$. From the measured data, we attempt to infer information on the optical properties of the underlying medium, for instance, the coefficients $\sigma_a$, $\sigma_s$ and $\Xi$. This inverse problem has been extensively investigated in the past decade, from mathematical, computational as well as practical perspectives; see for instance~\cite{AgKuKu-PIS09,AkBeDaElLiMi-JIIP17,AmBrJuWa-LNM12,BaBaVaRo-JOSA08,BuMaHaPa-PRE07,CoArBe-IP07,CoArKoBe-AO06,CoBe-JOSA09,FiHaRa-SIAM07,Haltmeier-SIAM11,HaScBuPa-IP04,HaScSc-M2AS05,HrKuNg-IP08,KuKu-EJAM08,Kunyansky-IP07,LaCoZhBe-AO10,Nguyen-IPI09,PaNuBu-PMB09,PaSc-IP07,QiStUhZh-SIAM11,ReZh-SIAM18,RiNt-PRE05, StUh-IP09,TaLi-OE10,TrZhCo-IP10,XuWaAmKu-MP04,YuWaJi-OE07,Zemp-AO10} and references therein.

Due to the fact that the radiative transport model~\eqref{EQ:RTE} is computationally expensive to solve, simplified mathematical models are often preferred as the model for light propagation in PAT. The diffusion approximation to RTE, see~\eqref{EQ:Diff}, is the most commonly-used replacement model~\cite{BaUh-IP10,BaRe-IP11,Zemp-AO10}. While the diffusion approximation is much simpler for mathematical analysis and computational solution, it often suffers in terms of accuracy, especially in regions close to the source locations, a phenomenon that has been addressed extensively in the literature of optical tomography~\cite{Arridge-IP99,TaVaKoArKa-PMB05}, as well as PAT~\cite{TaCoKaAr-IP12}. In this work, we study the PAT inverse problem with the simplified $P_2$ equations~\cite{AdLa-PNE02,LaThKlSeGo-JCP02}, a more sophisticated approximation than the classical diffusion model to the radiative transport equation, as the model of light propagation. We show numerically that image reconstructions based on the simplified $P_2$ model, while computationally less expensive than RTE-based reconstructions, can be more accurate than those based on the classical diffusion model under right circumstances.

We make the following general assumptions in the rest of the paper: 

($\cA$-i) the domain $\Omega$ is simply-connected with smooth boundary $\partial\Omega$; ($\cA$-ii) the physical coefficients $(c, \Xi, \sigma_a, \sigma_s)$ are positive and bounded in the sense that $0<\underline{\alpha} \le c, \Xi, \sigma_a, \sigma_s \le \overline{\alpha} <\infty$ for some positive constants $\underline{\alpha}$ and $\overline{\alpha}$; ($\cA$-iii) the ultrasound speed function $c(\bx)$ and the Gr\"uneisen coefficient $\Xi(\bx)$ are smooth in $\Omega$; and ($\cA$-iv) the values of the coefficients are known on the boundary $\partial\Omega$.

Let us remark that the only non-trivial assumption here is ($\cA$-iv). It has been shown that~\cite{ReGaZh-SIAM13}, without this assumption, mainly the part that ${\sigma_s}_{|\partial\Omega}$ is known, the quantitative reconstruction in PAT is not possible due to non-uniqueness. Additional optical current data on $\partial\Omega$ can help recover the uniqueness of quantitative reconstructions~\cite{ReGaZh-SIAM13}, but will make the current paper un-necessarily complicated. We therefore make assumption ($\cA$-iv) to simplify the overall setup to focus on the main point of our work.

The rest of the paper is structured as follows. We first review briefly in Section~\ref{SEC:Model} the simplified $P_2$ approximation to the transport equation. We then present in Section~\ref{SEC:Alg} the main computational reconstruction algorithm we use for our numerical studies. In Section~\ref{SEC:Theory} we study the quantitative step of the inverse problems based on the simplified $P_2$ equations. Detailed numerical simulation results based on synthetic data are presented in Section~\ref{SEC:Num} to demonstrate the feasibility of our approach. Concluding remarks are offered in Section~\ref{SEC:Concl}.

\section{The simplified $P_2$ approximation}
\label{SEC:Model}

We now introduce the simplified $P_2$ approximation to the radiative transport equation in diffusive regimes. To limit the length of the paper, we will only present the simplified $P_2$ equations. We refer interested readers to~\cite{FrKlLaYa-JCP07,KlLa-JCP06,LaThKlSeGo-JCP02,WrScAr-MST07} and the references therein for details on the derivation and numerical validation of the simplified $P_2$ equations. 

We first define the following sequence of total absorption coefficients:
\begin{equation}\label{EQ:Total Abso}
\sigma_{an}(\bx)=\sigma_a(\bx)+(1-g^n)\sigma_s(\bx),\qquad n\ge 0 .
\end{equation}
We also define the diffusion coefficients:
\begin{equation}\label{EQ:D Coeff}
D(\bx)=\dfrac{1}{3\sigma_{a1}(\bx)},\qquad \mbox{and}\qquad \wt D(\bx)=\dfrac{1}{7\sigma_{a3}(\bx)}.
\end{equation}
The simplified $P_2$ equations, together with its boundary conditions, take the following form~\cite{KlLa-JCP06,LaThKlSeGo-JCP02,WrScAr-MST07}:
\begin{equation}\label{EQ:SP2}
\begin{array}{rcll}
-\nabla\cdot D \nabla \phi_1(\bx) +\sigma_a \phi_1(\bx)-\frac{2}{3}\sigma_a \phi_2(\bx) &=& 0, & \mbox{in}\ \Omega\\
-\nabla\cdot \wt D \nabla \phi_2(\bx) +(\frac{5}{9}\sigma_{a2}+\frac{4}{9}\sigma_a) \phi_2(\bx)-\frac{2}{3}\sigma_a \phi_1(\bx) & = & 0, & \mbox{in}\ \Omega\\
\bn\cdot D\nabla\phi_1 +\frac{1}{2}\phi_1-\frac{1}{8}\phi_2 &=& \frac{1}{2} f(\bx), & \mbox{on}\ \partial\Omega\\
\bn\cdot \wt D\nabla\phi_2 +\frac{7}{24}\phi_2-\frac{1}{8}\phi_1 & = & -\frac{1}{8} f(\bx), & \mbox{on}\ \partial\Omega
\end{array}
\end{equation}
where $\phi_1$ and $\phi_2$ are the first two composites, i.e. linear combinations of, Legendre moments of the transport solution $u(\bx,\bv)$; see~\cite{FrKlLaYa-JCP07,KlLa-JCP06,LaThKlSeGo-JCP02,WrScAr-MST07} for more details.

The initial pressure field generated due to the photoacoustic effect, corresponding to ~\eqref{EQ:Int Data RTE}, can be written as follows in the simplified $P_2$ approximation:
\begin{equation}\label{EQ:Int Data SP2}
H(\bx)=H^{P_2}(\bx) \equiv \Xi(\bx) \sigma_a(\bx) (\phi_1-\dfrac{2}{3}\phi_2) .
\end{equation}

To simplify the analysis, we assume in the following that the absorption coefficient $\sigma_a$ is very small compared to the effective scattering coefficient $(1-g)\sigma_s$, i.e. $\sigma_a\ll (1-g)\sigma_s$. Under this assumption, we can neglect the factor $\sigma_a$ in the definitions of the diffusion coefficients $(D, \wt D)$ and the absorption coefficients $\sigma_{an}$ such that
\begin{equation}\label{EQ:Coeff Ass}
\begin{array}{c}
\sigma_{a0}=\sigma_a,\qquad\ \ \ \sigma_{an} =(1-g^n)\sigma_s,\ n\ge 1,\qquad\ \ \ D(\bx)=\dfrac{1}{3(1-g)\sigma_s(\bx)},\\
\wt D(\bx)=\dfrac{3}{7(1+g+g^2)}D(\bx), \qquad \ \ \sigma_{a2}=\dfrac{1+g}{3}\dfrac{1}{D}.
\end{array}
\end{equation}
We note that this assumption is generally required for any diffusive approximations to the radiative transport equation. Moreover, in applications of PAT for biological tissues, $\frac{\sigma_a}{(1-g)\sigma_s}$ is on the order of $10^{-2}\ll 1$~\cite{Jacques-PMB13}.

With the simplification in~\eqref{EQ:Coeff Ass}, the simplified $P_2$ system~\eqref{EQ:SP2} reduces to the following form:
\begin{equation}\label{EQ:SP2 simp}
\begin{array}{rcll}
-\nabla\cdot D \nabla \phi_1 +\sigma_a \phi_1 - \frac{2}{3}\sigma_a \phi_2 &=& 0, & \mbox{in}\ \Omega\\
-\nabla\cdot D \nabla \phi_2 + (\frac{5}{9\kappa'}\frac{1}{D}+\frac{4}{9\kappa}\sigma_a) \phi_2 -\frac{2}{3\kappa}\sigma_a \phi_1 &=& 0, & \mbox{in}\ \Omega\\
\bn\cdot D \nabla \phi_1 + \frac{1}{2}\phi_1 - \frac{1}{8} \phi_2 & = & \frac{1}{2}f(\bx), & \mbox{on}\ \partial\Omega\\
\bn\cdot D \nabla\phi_2 +\frac{7}{24\kappa}\phi_2 - \frac{1}{8\kappa}\phi_1 & = & -\frac{1}{8\kappa} f(\bx), & \mbox{on}\ \partial\Omega
\end{array}
\end{equation}
where
\[
 	\kappa = \dfrac{3}{7(1+g+g^2)}, \qquad \mbox{and}\qquad \kappa' = \dfrac{3}{1+g} \kappa.
\]
The corresponding initial pressure field $H^{P_2}$ in~\eqref{EQ:Int Data SP2} remains in the same form.

It is sometimes more convenient to rewrite the simplified $P_2$ system~\eqref{EQ:SP2 simp} in a new pair of variables $(\varphi_1, \varphi_2)$: $\varphi_1=\phi_1-\dfrac{2}{3}\phi_2$ and $\varphi_2=\phi_2$. In this case, we have
\begin{equation}\label{EQ:SP2 simp V2}
\begin{array}{rcll}
-\nabla\cdot D \nabla \varphi_1 +(1+\frac{4}{9\kappa})\sigma_a \varphi_1 - \frac{10}{27\kappa'}\frac{1}{D} \varphi_2 &=& 0, & \mbox{in}\ \Omega\\
-\nabla\cdot D \nabla \varphi_2 + \frac{5}{9\kappa'}\frac{1}{D} \varphi_2 -\frac{2}{3\kappa}\sigma_a \varphi_1 &=& 0, & \mbox{in}\ \Omega\\
\bn\cdot D \nabla \varphi_1 +\frac{6\kappa+1}{12\kappa} \varphi_1 +\frac{15\kappa-10}{72\kappa} \varphi_2 & = & \frac{6\kappa+1}{12\kappa}f(\bx), & \mbox{on}\ \partial\Omega\\
\bn\cdot D \nabla \varphi_2 +\frac{5}{24\kappa} \varphi_2 - \frac{1}{8\kappa} \varphi_1 & = & -\frac{1}{8\kappa} f(\bx), & \mbox{on}\ \partial\Omega
\end{array}
\end{equation}
The corresponding initial pressure field can now be written as
\begin{equation}\label{EQ:Int Data SP2 V2}
H(\bx) = H^{P_2'}(\bx) \equiv \Xi(\bx) \sigma_a(\bx) \varphi_1(\bx).
\end{equation}

To recover the classical diffusion approximation to the radiative transport equation, we drop the terms involve gradient of $u_2$ from the simplified $P_2$ system~\eqref{EQ:SP2 simp V2}. This leads to the simplified $P_1$ approximation:
\begin{equation}\label{EQ:Diff}
\begin{array}{rcll}
-\nabla\cdot D(\bx) \nabla \varphi(\bx) +\sigma_a(\bx) \varphi(\bx)& = & 0, & \mbox{in}\ \Omega\\
\bn \cdot D\nabla \varphi + \frac{1}{2} \varphi & = & \frac{1}{2} f(\bx), & \mbox{on}\ \partial\Omega
\end{array}
\end{equation}
The corresponding initial pressure field generated in this case takes the form:
\begin{equation}\label{EQ:Int Data Diff}
H(\bx)=H^{diff}(\bx) \equiv \Xi(\bx)\sigma_a(\bx) \varphi(\bx).
\end{equation}
Note that $\varphi_1$ in~\eqref{EQ:SP2 simp V2} and $\varphi$ in~\eqref{EQ:Diff} are different approximations to the same physical quantity, the photon density. We use different symbols here for the two to avoid unnecessary confusion. Moreover, the boundary condition in the diffusion approximation~\eqref{EQ:Diff}, i.e. simplified $P_1$ approximation, is slightly different from the boundary condition one can obtain from a detailed boundary layer analysis in the classical $P_1$ approximations~\cite{DaLi-Book93-6}. For the sake of consistency, we use the boundary condition in~\eqref{EQ:Diff}.

It is well-known that under reasonable regularity assumptions on the optical coefficients $(\sigma_a, \sigma_s)$ and the boundary illumination source $f$, the radiative transport equation~\eqref{EQ:RTE} and its diffusion approximation~\eqref{EQ:Diff} are well-posed in appropriate function spaces~\cite{DaLi-Book93-6}. Therefore the initial pressure fields $H^{rte}$ and $H^{diff}$ are well-defined quantities. The simplified $P_2$ system is also well-posed under similar assumptions. In fact, let $\boldsymbol \phi=(\phi_1, \phi_2)$ and $\boldsymbol\varphi=(\varphi_1, \varphi_2)$, we can verify that the bilinear form associated with~\eqref{EQ:SP2 simp},
\begin{multline*}
	\cB(\boldsymbol \phi, \boldsymbol \varphi) = \int_\Omega \left[D\nabla\phi_1\cdot\nabla\varphi_1+\kappa D\nabla\phi_2\cdot\nabla\varphi_2\right] d\bx + \int_\Omega \left[ \sigma_a \phi_1\varphi_1+(\frac{5(1+g)}{27D}+\dfrac{4}{9}\sigma_a)\phi_2\varphi_2\right] d\bx\\
	-\int_\Omega \dfrac{2}{3}\sigma_a \left[ \phi_2\varphi_1+\phi_1\varphi_2 \right] d\bx +\int_{\partial\Omega} \left[\dfrac{1}{2}\phi_1\varphi_1+\dfrac{7}{24}\phi_2\varphi_2-\dfrac{1}{8}(\phi_1\varphi_2+\phi_2\varphi_1)\right] dS(\bx),
\end{multline*}
is coercive in the classical sense, due to the assumptions on the coefficients and the fact that $2\sqrt{\sigma_a}\sqrt{\frac{5(1+g)}{27D}+\dfrac{4}{9}\sigma_a}> \dfrac{4}{3}\sigma_a$ and $2 \sqrt{1/2}\sqrt{7/24}>2/8$. The Lax-Milgram theorem and standard theory of elliptic systems~\cite{GiTr-Book00,McLean-Book00} can then be applied to the simplified $P_2$ system~\eqref{EQ:SP2 simp}. Therefore, $H^{P_2}$ is also a well-defined quantity.

\section{Numerical reconstruction algorithms}
\label{SEC:Alg}

In this section, we implement a standard optimal control-based numerical image reconstruction algorithm for PAT with the simplified $P_2$ equations~\eqref{EQ:SP2 simp} and the classical diffusion equation~\eqref{EQ:Diff} as the models of light propagation. Since the data we have are not enough to uniquely determine all three coefficients $(\Xi, \sigma_a, \sigma_s)$ simultaneously, as indicated in both the classical diffusion and radiative transport regimes~\cite{BaRe-IP11,MaRe-CMS14}, we consider here only the reconstruction of two coefficients. We provide the details for the case of reconstructing $(\sigma_a, \sigma_s)$. Reconstructing other pairs, such as $(\Xi, \sigma_a)$, can be done very similarly.

Let us assume that we have data collected from $J\ge 2$ different illuminations $\{f_j\}_{j=1}^J$. We denote by $\{p_j^*\}_{j=1}^J$ the measured ultrasound data. We solve the reconstruction problem by searching for the coefficient pair $(\sigma_a, \sigma_s)$ that minimize the mismatch between ultrasound data predicted by the mathematical models and the measurements. More precisely, we solve the minimization problem
\begin{equation}\label{EQ:Min}
\min_{\sigma_a, \sigma_s}\cO(\sigma_a, \sigma_s),\ \ \mbox{subject to}, \ \ \mathfrak{l}_a \le \sigma_a \le \mathfrak{u}_a, \ \ \mathfrak{l}_s \le \sigma_s \le \mathfrak{u}_s
\end{equation}
where the linear bounds $\{\mathfrak{l}_a, \mathfrak{u}_a, \mathfrak{l}_s, \mathfrak{u}_s\}$ are selected in a case by case manner, as discussed further in the numerical simulations in Section~\ref{SEC:Num}. The data mismatch functional is defined as
\begin{equation}\label{EQ:Obj}
\cO(\sigma_a, \sigma_s) =\dfrac{1}{2} \sum_{j=1}^J \int_0^T \int_{\partial\Omega}(p_j^{\cM} -p_j^*)^2 dS(\bx)dt + \alpha\cR(\sigma_a)+\beta\cR(\sigma_s),
\end{equation}
where $p_j^\cM$ is the ultrasound signal predicted using the light propagation model $\cM \in\{P_2, \mbox{diff}\}$ with the coefficient $(\sigma_a, \sigma_s)$. The parameters $\alpha$ and $\beta$ are used to control the strengths of the regularization mechanism encoded in the functional $\cR$. The regularization functional we select here is of Tikhonov type~\cite{EnHaNe-Book96}, based on the $L^2$ norm of the gradients,
\begin{equation}\label{EQ:Reg}
\cR(\sigma_a)=\dfrac{1}{2} \|\nabla \sigma_a\|_{(L^2(\Omega))^3}^2\equiv \dfrac{1}{2}\int_\Omega |\nabla \sigma_a|^2 d\bx.
\end{equation}
While we choose the same regularization functional for both $\sigma_a$ and $\sigma_s$ for convenience, it is not required. Other types of regularization can also be considered, but will not be discussed in this paper. Interested readers are referred to~\cite{EnHaNe-Book96,ScGrGrHaLe-Book09} for classical treatments of regularization theory and its applications.

To solve the minimization problem~\eqref{EQ:Min}, we use the SNOPT algorithm developed in~\cite{GiMuSa-SIAM05}. In a nutshell, this is a sparse sequential quadratic programming (SQP) algorithm where the Hessian of the Lagrangian is approximated by a limited-memory BFGS strategy. This is a mature optimization technique, therefore we will not describe it in detail here. Our main objective is to supply the optimization algorithm with a subroutine to evaluate the mismatch functional $\cO$ and its derivatives with respect to the optical properties $\sigma_a$ and $\sigma_s$. The derivatives can be computed in a standard manner using the adjoint state method. We summarize the calculations of the derivatives in the following lemma.
\begin{lemma}\label{LMMA:Derivative}
	Let $\Omega$, $c$ and $\Xi$ satisfy the assumptions in ($\cA$-i)-($\cA$-iv) and assume that the assumptions in~\eqref{EQ:Coeff Ass} hold as well. For each $1\le j\le J$, let $f_j(\bx)$ be the restriction of a $\cC^1$ function to $\partial\Omega$. Then the predicted ultrasound data for illumination source $f_j$ using specified optical model $\cM$, ${p_j^\cM}_{|(0, T]\times\partial\Omega}$, viewed as the map:
	\begin{equation}
	p_j^\cM:
	\begin{array}{ccl}
	(\sigma_a, \sigma_s) & \mapsto& {p_j^\cM}_{|(0,T]\times\partial\Omega}\\
	\cC^1(\bar\Omega)\times \cC^1(\bar \Omega) &\mapsto & \cH^{1/2}((0,T]\times\partial\Omega)
	\end{array}
	\end{equation}
	is Fr\'echet differentiable at any $(\sigma_a, \sigma_s)\in \cC^1(\bar\Omega)\times \cC^1(\bar\Omega)$ that satisfies the assumptions in ($\cA$-ii). Moreover, the mismatch functional $\cO(\sigma_a, \sigma_s): \cC^1(\bar\Omega)\times \cC^1(\bar \Omega) \mapsto \bbR_+$ given by \eqref{EQ:Obj} is Fr\'echet differentiable and its derivatives at $(\sigma_a,\sigma_s)$ in the directions $\delta\sigma_a\in \cC_0^1(\bar\Omega)$ and $\delta\sigma_s\in \cC_0^1(\Omega)$ (such that $\sigma_a+\delta\sigma_a$ and $\sigma_s+\delta\sigma_s$ satisfy ($\cA$-ii)) are given as follows. Let $q_j(t,\bx)$, $1\le j\le J$, be the solution to the adjoint wave equation:
	\begin{equation}\label{EQ:Wave Adj}
	\begin{array}{rcll}
	\dfrac{1}{c^2(\bx)}\dfrac{\partial^2 q_j}{\partial t^2} -\Delta q_j &=& 0, 
	&\text{in}\ (0, T)\times \Omega \\
	q_j(t,\bx)= 0,\ \  
	\dfrac{\partial q_j}{\partial t}(t,\bx) &=& 0,&\mbox{in}\ \{T\}\times\Omega\\
	\bn\cdot \nabla q_j(t, \bx) &=& p_j^\cM-p_j^*, & \mbox{on}\ (0, T)\times\partial\Omega
	\end{array}
	\end{equation}
	(i) If the optical model $\cM$ is the simplified $P_2$ model~\eqref{EQ:SP2 simp}, then
	\begin{eqnarray}
	\nonumber
	\cO'(\sigma_a, \sigma_s)[\delta\sigma_a] &= &\sum_{j=1}^J\int_\Omega (\phi_{1,j}-\dfrac{2}{3}\phi_{2,j}) \left\{\dfrac{\Xi}{c^2} \dfrac{\partial q_j}{\partial t}(0, \bx) + (\psi_{1,j}-\dfrac{2}{3\kappa}\psi_{2,j})\right\}\delta\sigma_a(\bx) d\bx\\ 
	\label{EQ:Obj Deriv SP2 a} & & + \alpha\cR'(\sigma_a)[\delta\sigma_a],\\
	\nonumber
	\cO'(\sigma_a, \sigma_s)[\delta\sigma_s] &=&\sum_{j=1}^J\int_\Omega\left\{ \dfrac{\nabla\psi_{1,j}\cdot\nabla\psi_{1,j}+\nabla\phi_{2,j}\cdot\nabla\psi_{2,j}}{3(1-g)\sigma_s^2} - \dfrac{5(1-g^2)}{9\kappa}\phi_{2,j}\psi_{2,j} \right\} \delta\sigma_s(\bx) d\bx \\ 
	\label{EQ:Obj Deriv SP2 b} & & +\beta\cR'(\sigma_s)[\delta\sigma_s],
	\end{eqnarray}
	where $(\phi_{1,j}, \phi_{2,j})$ solves~\eqref{EQ:SP2 simp} with source $f_j$, $1\le j\le J$, while $(\psi_{1,j}, \psi_{2,j})$ solves the adjoint diffusion system:
	\begin{equation}\label{EQ:SP2 simp Adj}
	\begin{array}{rcll}
	-\nabla\cdot D \nabla \psi_{1,j} +\sigma_a (\psi_{1,j} - \dfrac{2}{3\kappa}\psi_{2,j}) &=& - \dfrac{\Xi}{c^2}\sigma_a\partial_t q_j(0,\bx), & \mbox{in}\ \Omega\\
	-\nabla\cdot D \nabla \psi_{2,j} + (\dfrac{5}{9\kappa'}\dfrac{1}{D}+\dfrac{4}{9\kappa}\sigma_a) \psi_{2,j} -\dfrac{2}{3}\sigma_a \psi_{1,j} &=& \dfrac{2}{3}\dfrac{\Xi}{c^2}\sigma_a\partial_t q_j(0,\bx), & \mbox{in}\ \Omega\\
	\bn\cdot D \nabla \psi_{1,j} + \dfrac{1}{2}\psi_{1,j} - \dfrac{1}{8\kappa} \psi_{2,j} & = & 0, & \mbox{on}\ \partial\Omega\\
	\bn\cdot D \nabla\psi_{2,j} +\dfrac{7}{24\kappa}\psi_{2,j} - \dfrac{1}{8}\psi_{1,j} & = & 0, & \mbox{on}\ \partial\Omega
	\end{array}
	\end{equation}
	(ii) If the optical model $\cM$ is the classical diffusion model~\eqref{EQ:Diff}, then
	\begin{eqnarray}
	\nonumber \cO'(\sigma_a, \sigma_s)[\delta\sigma_a] &=& \sum_{j=1}^J\int_\Omega \varphi_j \left\{\dfrac{\Xi}{c^2} \dfrac{\partial q_j}{\partial t}(0, \bx)+\eta_j\right\}\delta\sigma_a(\bx) d\bx\\
	\label{EQ:Obj Deriv P1 a} & & {\hskip 3cm} +\alpha\cR'(\sigma_a)[\delta\sigma_a], \\
	\label{EQ:Obj Deriv P1 b}
	\cO'(\sigma_a, \sigma_s)[\delta\sigma_s] &=&\sum_{j=1}^J\int_\Omega \dfrac{\nabla \varphi_j\cdot\nabla \eta_j}{3(1-g)\sigma_s^2} \delta \sigma_s(\bx) d\bx +\beta\cR'(\sigma_s)[\delta\sigma_s],
	\end{eqnarray}
	where $\varphi_j$ solves~\eqref{EQ:Diff} with source $f_j$, $1\le j\le J$, and $\eta_j$ solves
	\begin{equation}\label{EQ:Diff Adj}
	\begin{array}{rcll}
	-\nabla\cdot D(\bx) \nabla \eta_j(\bx) +\sigma_a(\bx) \eta_j(\bx)& = & - \dfrac{\Xi}{c^2}\sigma_a\partial_t q_j(0,\bx), & \mbox{in}\ \Omega\\
	\bn \cdot D\nabla \eta_j + \dfrac{1}{2} \eta_j & = &0, & \mbox{on}\ \partial\Omega
	\end{array}
	\end{equation}
\end{lemma}
\begin{proof}
	The result on the classical diffusion model~\eqref{EQ:Diff} is proved in~\cite{DiReVa-IP15}. We focus here on the simplified $P_2$ model. Under the assumptions stated in the lemma, standard elliptic theory in~\cite{GiTr-Book00,McLean-Book00} implies that~\eqref{EQ:SP2 simp} admits a unique solution pair $(\phi_{1,j}, \phi_{2,j})\in\cH^2(\Omega)\times\cH^2(\Omega)$ for source $f_j$. This, together with the assumption on $\sigma_a$, gives that the initial pressure field $H_j^{\cM}\in \cH^1(\Omega)$, which then ensures that the wave equation~\eqref{EQ:Wave} admits a unique solution $p_j^\cM\in\cH^1((0, T]\times\Omega)$~\cite{DiDoNaPaSi-SIAM02,Isakov-Book02,Rakesh-CPDE88}. Differentiability of ${p_j^\cM}_{|(0, T]\times\partial\Omega}$ with respect to the initial pressure field $H_j^{\cM}$ then follows from the linearity of the map $p_j^{\cM}: \cH^1(\Omega)\to \cH^{1/2}((0,T]\times\partial\Omega)$~\cite{DiDoNaPaSi-SIAM02,Isakov-Book02}. It remains, following the chain rule, to show that $H_j^{\cM}: \cC^1(\bar\Omega)\times\cC^1(\bar\Omega) \mapsto \cH^1(\Omega)$ is Fr\'echet differentiable with respect to $\sigma_a$ and $\sigma_s$. We now prove this for the derivative with respect to $\sigma_a$. The other derivative follows from similar calculations. 
	
	Let $\big(\phi_{1,j}^{(\sigma_a+\delta\sigma_a, \sigma_s)}, \phi_{2,j}^{(\sigma_a+\delta\sigma_a, \sigma_s)}\big)$ and $\big(\phi_{1,j}^{(\sigma_a, \sigma_s)}, \phi_{2,j}^{(\sigma_a, \sigma_s)}\big)$ be the solution to the simplified $P_2$ model~\eqref{EQ:SP2 simp} with coefficients $(\sigma_a+\delta\sigma_a, \sigma_s)$ and $(\sigma_a, \sigma_s)$ respectively, for source $f_j$. Let $(\wt\phi_{1,j}, \wt\phi_{2,j})$ be the solution to,
	\begin{equation}\label{EQ:SP2 simp Deriv}
	\begin{array}{rcll}
	-\nabla\cdot D \nabla \wt\phi_{1,j} +\sigma_a (\wt\phi_{1,j} - \frac{2}{3}\wt\phi_{2,j}) &=& -\Delta\phi\delta\sigma_a, & \mbox{in}\ \Omega\\
	-\nabla\cdot D \nabla \wt\phi_{2,j} + \frac{5}{9\kappa'}\frac{1}{D} \wt\phi_{2,j} -\frac{2\sigma_a}{3\kappa} (\wt\phi_{1,j}-\frac{2}{3}\wt\phi_{2,j}) &=& \frac{2}{3\kappa}\Delta\phi\delta\sigma_a, & \mbox{in}\ \Omega\\
	\bn\cdot D \nabla \wt\phi_{1,j} + \frac{1}{2}\wt\phi_{1,j} - \frac{1}{8} \wt \phi_{2,j} & = & 0, & \mbox{on}\ \partial\Omega\\
	\bn\cdot D \nabla\wt\phi_{2,j} +\frac{7}{24\kappa}\wt\phi_{2,j} - \frac{1}{8\kappa}\wt\phi_{1,j} & = &0, & \mbox{on}\ \partial\Omega
	\end{array}
	\end{equation}
	where we  have used the notation $\Delta\phi=\phi_{1,j}^{(\sigma_a, \sigma_s)}-\frac{2}{3}\phi_{2,j}^{(\sigma_a, \sigma_s)}$.
	We define $\wh \phi_{i,j}=\phi_{i,j}^{(\sigma_a+\delta\sigma_a, \sigma_s)}-\phi_{i,j}^{(\sigma_a, \sigma_s)}$ and $\doubleline{\phi}_{i,j}=\wh\phi_{i,j}-\wt\phi_{i,j}$, $i=1,2$. It is straightforward to check that $(\wh\phi_{1,j}, \wh\phi_{2,j})$ solves
	\begin{equation}\label{EQ:W Diff}
	\begin{array}{rcll}
	-\nabla\cdot D \nabla \wh\phi_{1,j}+(\sigma_a+\delta\sigma_a) (\wh\phi_{1,j} - \frac{2}{3}\wh\phi_{2,j}) &=& -\Delta\phi \delta\sigma_a, & \mbox{in}\ \Omega\\
	-\nabla\cdot D \nabla \wh\phi_{2,j} + \frac{5}{9\kappa'}\frac{1}{D}\wh\phi_{2,j} -\frac{2(\sigma_a+\delta\sigma_a)}{3\kappa} (\wh\phi_{1,j}-\frac{2}{3}\wh\phi_{2,j}) &=& \frac{2}{3\kappa}\Delta\phi\delta\sigma_a, & \mbox{in}\ \Omega\\
	\bn\cdot D \nabla \wh\phi_{1,j} + \frac{1}{2}\wh\phi_{1,j} - \frac{1}{8} \wh \phi_{2,j} & = & 0, & \mbox{on}\ \partial\Omega\\
	\bn\cdot D \nabla\wh\phi_{2,j}+\frac{7}{24\kappa}\wh\phi_{2,j} - \frac{1}{8\kappa}\wh\phi_{1,j} & = &0, & \mbox{on}\ \partial\Omega
	\end{array}
	\end{equation}
	and $(\doubleline{\phi}_{1,j}, \doubleline{\phi}_{2,j})$ solves
	\begin{equation}\label{EQ:Z Diff}
	\begin{array}{rcll}
	-\nabla\cdot D \nabla\doubleline{\phi}_{1,j} + \sigma_a (\doubleline{\phi}_{1,j} - \frac{2}{3}\doubleline{\phi}_{2,j}) &=& - (\wh\phi_{1,j}-\frac{2}{3}\wh\phi_{2,j})\delta\sigma_a, & \mbox{in}\ \Omega\\
	-\nabla\cdot D \nabla \doubleline{\phi}_{2,j} + \frac{5}{9\kappa'}\frac{1}{D}\doubleline{\phi}_{2,j} -\frac{2}{3\kappa}\sigma_a (\doubleline{\phi}_{1,j}-\frac{2}{3}\doubleline{\phi}_{2,j}) &=&  \frac{2}{3\kappa}(\wh\phi_{1,j}-\frac{2}{3}\wh\phi_{2,j})\delta\sigma_a, & \mbox{in}\ \Omega\\
	\bn\cdot D \nabla \doubleline{\phi}_{1,j} + \frac{1}{2}\doubleline{\phi}_{1,j} - \frac{1}{8} \doubleline{\phi}_{2,j} & = & 0, & \mbox{on}\ \partial\Omega\\
	\bn\cdot D \nabla\doubleline{\phi}_{2,j} +\frac{7}{24\kappa}\doubleline{\phi}_{2,j} - \frac{1}{8\kappa}\doubleline{\phi}_{1,j} & = &0, & \mbox{on}\ \partial\Omega
	\end{array}
	\end{equation}
	Note that in the above derivations, we have used the assumption that the boundary value of the coefficient $\sigma_a$ is known, i.e. ${\delta\sigma_a}_{|\partial\Omega}=0$; see the assumption in ($\cA$-iv) and $\delta\sigma_a\in \cC_0^1(\Omega)$.
	
	We first observe from the simplified $P_2$ model~\eqref{EQ:SP2 simp}, following standard elliptic theory~\cite{Evans-Book10,GiTr-Book00,Ladyzhenskaya-Book85,McLean-Book00}, that for $0\le k\le 2$,
	\[
	\|(\phi_{1,j}, \phi_{2,j})\|_{[\cH^k(\Omega)]^2} \le \mathfrak{c}_1 \|f_j\|_{L^2(\partial\Omega)}.
	\]
	In the same way, equation~\eqref{EQ:W Diff} admits a unique solution with
	\begin{equation}\label{EQ:Stab WV}
	\|(\wh\phi_{1,j}, \wh \phi_{2,j})\|_{[\cH^k(\Omega)]^2} \le \mathfrak{c}_2 \|\delta\sigma_{a}(\phi_{1,j}, \phi_{2,j})\|_{[L^2(\Omega)]^2}  \le \wt{\mathfrak{c}}_2 \|\delta\sigma_{a}\|_{\cC_0^1(\Omega)}\|(\phi_{1,j}, \phi_{2,j})\|_{[L^2(\Omega)]^2},
	\end{equation}
	while~\eqref{EQ:Z Diff} admits a unique solution satisfying
	\begin{equation}\label{EQ:Stab Z}
	\|(\doubleline{\phi}_{1,j}, \doubleline{\phi}_{2,j})\|_{[\cH^k(\Omega)]^2} \le \mathfrak{c}_3 \|\delta\sigma_{a}(\wh\phi_{1,j}, \wh \phi_{2,j})\|_{[L^2(\Omega)]^2} \le \wt{\mathfrak{c}}_3 \|\delta\sigma_{a}\|_{\cC_0^1(\Omega)}\|(\wh\phi_{1,j}, \wh \phi_{2,j})\|_{[L^2(\Omega)]^2}.
	\end{equation}
	We then deduce that
	\begin{equation}
	\|(\doubleline{\phi}_{1,j}, \doubleline{\phi}_{2,j})\|_{[\cH^k(\Omega)]^2} \le \mathfrak{c}_4 \|\delta\sigma_a\|_{\cC_0^1(\Omega)}^2\|f_j\|_{L^2(\partial\Omega)},
	\end{equation}
	which then leads to,
	\begin{equation}\label{EQ:Stab Lim}
	\lim_{\|\delta\sigma_a\|_{\cC_0^1(\Omega)}\to 0}\dfrac{\|(\doubleline{\phi}_{1,j},\doubleline{\phi}_{2,j})\|_{[\cH^k(\Omega)]^2}}{\|\delta\sigma_a\|_{\cC_0^1(\Omega)}} \equiv \lim_{\|\delta\sigma_a\|_{\cC_0^1(\Omega)}\to 0}\dfrac{\|(\wh\phi_{1,j}-\wt\phi_{1,j}, \wh\phi_{2,j}-\wt\phi_{2,j})\|_{[\cH^k(\Omega)]^2}}{\|\delta\sigma_a\|_{\cC_0^1(\Omega)}}=0.
	\end{equation}
	This shows that the map $(\phi_{1,j}, \phi_{2,j})$ is Fr\'echet differentiable with respect to $\sigma_a$, as a map $(\phi_{1,j}, \phi_{2,j}): \cC^1(\bar\Omega) \mapsto \cH^k(\Omega)\times \cH^k(\Omega)$ ($0\le k\le 2$), with Fr\'echet derivative in direction $\delta\sigma_a\in\cC_0^1(\Omega)$ given by $(\wt\phi_{1,j}, \wt\phi_{2,j})$. Differentiability of $H_j^\cM$ with respect to $\sigma_a$ then follows from this fact, the regularity and boundedness of $\sigma_a$ and $(\phi_{1,j}, \phi_{2,j})$, and the product rule. 
	
	To compute the Fr\'echet derivative of $\cO(\sigma_a, \sigma_s)$ with respect to $\sigma_a$, we first compute
	\[
	{H_j^{\cM}}'(\sigma_a, \sigma_s)[\delta\sigma_a]=\Xi (\phi_{1,j} - \frac{2}{3}\phi_{2,j}) \delta\sigma_a + \Xi \sigma_a (\phi_{1,j} -\frac{2}{3}\phi_{2,j})'(\sigma_a, \sigma_s)[\delta\sigma_a],
	\]
	and
	\begin{equation}\label{EQ:O Deriv-a}
	\cO'(\sigma_a, \sigma_s)[\delta\sigma_a] = \sum_{j=1}^J\int_0^T\int_{\partial\Omega} (p_j^{\cM}-p_j^*){p_j^{\cM}}'(\sigma_a, \sigma_s)[\delta\sigma_a] dS(\bx) dt + \alpha\cR'(\sigma_a)[\delta\sigma_a].
	\end{equation}
	Let us denote $w_j := {p_j^{\cM}}'(\sigma_a, \sigma_s)[\delta\sigma_a]$. We verify that $w_j$ solves
	\begin{equation}\label{EQ:Wave Deriv}
	\begin{array}{rcll}
	\dfrac{1}{c^2(\bx)}\dfrac{\partial^2 w_j}{\partial t^2} -\Delta w_j &=& 0, 
	&\text{in}\ \bbR_+ \times \Omega \\
	w_j (t,\bx) = {H_j^{\cM}}'(\sigma_a, \sigma_s)[\delta\sigma_a],\ \  
	\dfrac{\partial w_j}{\partial t}(t,\bx) &=& 0,&\mbox{in}\ \{0\}\times\Omega\\
	\bn\cdot \nabla w_j &=& 0, & \mbox{on}\ \bbR_+\times\partial\Omega.
	\end{array}
	\end{equation}
	Multiplying the equation for $q_j$,~\eqref{EQ:Wave Adj}, by $w_j$, the equation for $w_j$,~\eqref{EQ:Wave Deriv}, by $q_j$, and integrating the difference over $(0, T)\times\Omega$, we arrive at
	\begin{multline}\label{EQ:O Deriv-b}
	\int_0^T\int_{\partial\Omega} (p_j^{\cM}-p_j^*) w_j dS(\bx) dt = \int_\Omega \dfrac{\Xi}{c^2} (\phi_{1,j}-\frac{2}{3}\phi_{2,j}) \dfrac{\partial q_j}{\partial t}(0, \bx) \delta\sigma_a d\bx \\
	+\int_\Omega  \dfrac{\Xi}{c^2} \sigma_a (\phi_{1,j}-\frac{2}{3}\phi_{2,j})'(\sigma_a, \sigma_s)[\delta\sigma_a] \dfrac{\partial q_j}{\partial t}(0, \bx) d\bx,
	\end{multline}
	thanks to Green's theorem. 
	
	Multiplying the equation for $(\psi_{1,j}, \psi_{2,j})$, i.e.~\eqref{EQ:SP2 simp Adj}, by $(\wt\phi_{1,j}, \wt\phi_{2,j})$ (which is nothing but $(\phi_{1,j}'(\sigma_a, \sigma_s)[\delta \sigma_a], \phi_{2,j}'(\sigma_a, \sigma_s)[\delta \sigma_a])$), the equation for $(\wt\phi_{1,j}, \wt\phi_{2,j})$, i.e.~\eqref{EQ:SP2 simp Deriv}, by $(\psi_{1,j}, \psi_{2,j})$, and integrating the difference over $\Omega$, we obtain
	\begin{multline}\label{EQ:O Deriv-c}
	\int_\Omega \dfrac{\Xi}{c^2} \sigma_a (\phi_{1,j}-\dfrac{2}{3}\phi_{2,j})'(\sigma_a, \sigma_s)[\delta\sigma_a] \dfrac{\partial q_j}{\partial t}(0, \bx) d\bx\\ =
	\int_\Omega (\phi_{1,j}-\dfrac{2}{3}\phi_{2,j})(\psi_{1,j}-\dfrac{2}{3\kappa}\psi_{2,j}) \delta\sigma_a d\bx.
	\end{multline}
	We now combine~\eqref{EQ:O Deriv-a}, ~\eqref{EQ:O Deriv-b} and ~\eqref{EQ:O Deriv-c} to get the final result in~\eqref{EQ:Obj Deriv SP2 a}. Similar calculations for $\sigma_s$ yield the result in~\eqref{EQ:Obj Deriv SP2 b}. This completes the proof.
\end{proof}

Let us emphasize here that the simplified $P_2$ diffusion system~\eqref{EQ:SP2 simp} is not self-adjoint. Therefore, the diffusion operators and the boundary conditions in~\eqref{EQ:SP2 simp Adj} are different from those in~\eqref{EQ:SP2 simp}. 

The calculations in Lemma~\ref{LMMA:Derivative} allow us to develop a subroutine for the SNOPT algorithm to evaluate the mismatch functional $\cO$ and its derivatives with respect to $\sigma_a$ and $\sigma_s$. For the convenience of presentation, let us denote by $\cO_j(\sigma_a, \sigma_s)$ the contributions to the mismatch functional $\cO$ from source $f_j$, that is, $\cO_j(\sigma_a, \sigma_s)=\dfrac{1}{2} \int_0^T \int_{\partial\Omega}(p_j^{\cM} -p_j^*)^2 dS(\bx)dt$. We use $\cO_j'$ to denote the derivative of $\cO_j$ at $(\sigma_a, \sigma_s)$. The algorithm for calculating $\cO$ and $\cO'$ is summarized in {\bf Subroutine}~\ref{ALG:Func Eval}.
\begin{algorithm}
	\caption{Evaluating $\cO$ and Its Derivatives at $(\sigma_a, \sigma_s)$ for Model $\cM$} 
	\begin{algorithmic}[1]
		\State Initialize $\cO=0$ and $\cO'=\bzero$
		\For{$j=1$ to $J$} 
		\State Solve the forward model $\cM$ (i.e., ~\eqref{EQ:SP2 simp} or ~\eqref{EQ:Diff}), with illumination source $f_j$
		\State Evaluate initial pressure field $H_j^{\cM}$ for model $\cM$ (following ~\eqref{EQ:Int Data SP2} or ~\eqref{EQ:Int Data Diff})
		\State Solve the wave equation~\eqref{EQ:Wave} with initial condition $H_j^\cM$ for $p_j^\cM$
		\State Evaluate the residual $z_j^\cM=p_j^{\cM}-p_j^*$ and $\cO_j=\frac{1}{2}\int_0^T \int_{\partial\Omega} (z_j^\cM)^2 dS(\bx) dt$ 
		\State $\cO \leftarrow \cO+\cO_j$
		\State Solve the adjoint wave equation~\eqref{EQ:Wave Adj}, and evaluate $\partial_t q_j(0,\bx)$
		\State Solve the adjoint diffusion equation for model $\cM$ (i.e.,~\eqref{EQ:SP2 simp Adj} or ~\eqref{EQ:Diff Adj})
		\State Evaluate the derivative $\cO_j'$ (following ~\eqref{EQ:Obj Deriv SP2 a} and ~\eqref{EQ:Obj Deriv SP2 b}, or~\eqref{EQ:Obj Deriv P1 a} and ~\eqref{EQ:Obj Deriv P1 b} ) 
		\State $\cO' \leftarrow \cO'+\cO_j'$
		\EndFor
		\State $\cO \leftarrow \cO+\alpha\cR(\sigma_a)+\beta\cR(\sigma_s)$
		\State $\cO' \leftarrow \cO'+\alpha\cR'(\sigma_a)[\delta\sigma_a]+\beta\cR'(\sigma_s)[\delta\sigma_s]$
	\end{algorithmic}
	\label{ALG:Func Eval}
\end{algorithm}

Our choice of using the SNOPT algorithm is quite random. Other nonlinear minimization algorithms can be used to solve our minimization problem here. In fact, in some of our numerical simulations in Section~\ref{SEC:Num}, we compared the SNOPT algorithm with a BFGS-based quasi-Newton method we implemented in~\cite{ReBaHi-SIAM06} and an augmented Lagrange method we implemented in~\cite{AbReHi-IP05}. We did not observe any consistent differences between results from different algorithms in terms of reconstruction quality.

Let us mention that the reconstruction algorithm we implemented here is based on the one-step approach: we reconstruct the optical coefficients directly from the measured ultrasound signal. This is the same approach that has been recently used in~\cite{BeBoHaPr-M2AS14,DiReVa-IP15,HaNeRa-IP15,PuCoArGoKaTa-IEEE16,ReTr-EJAM18}. This approach is useful for instance in cases where one can only measure ultrasound data in part of the domain boundary.

\section{Quantitative inversion with simplified $P_2$}
\label{SEC:Theory}

An alternative approach for PAT reconstruction is a two-step strategy: (i) to reconstruct the initial pressure field $H$ from measured ultrasound data; and then (ii) to reconstruct the optical coefficients from the reconstructed initial pressure field $H$. The first step involves only the acoustic model and is independent of the optical model, and reconstruction algorithms for this step have previously been developed in many scenarios~\cite{AgKuKu-PIS09,AmGoLe-IP10,BuMaHaPa-PRE07,CoArBe-IP07,HaScBuPa-IP04,Hristova-IP09,KuKu-EJAM08,QiStUhZh-SIAM11,StUh-IP09,TrZhCo-IP10,XuWaAmKu-MP04}. The second step of the reconstruction have been developed for both the diffusion model~\eqref{EQ:Diff}~\cite{BaRe-IP11,BaUh-IP10} and the radiative transport model~\eqref{EQ:RTE}~\cite{BaJoJu-IP10,HaNeNgRa-SIAM18,MaRe-CMS14,ReZhZh-IP15,SaTaCoAr-IP13}, but not the simplified $P_2$ model~\eqref{EQ:SP2 simp}, to our best knowledge. 

The objective of this section is to study the quantitative step of PAT with the simplified $P_2$ model: to reconstruct the optical coefficients from the initial pressure field data $H$ that one recovers from the ultrasound measurements. We assume again that we have data generated from $J\ge 1$ illumination sources. Let $(\phi_{1,j}, \phi_{2,j})$ be the solution to the simplified $P_2$ system~\eqref{EQ:SP2 simp} with source $f_j$. In the quantitative step, we wish to recover the optical coefficients from the data $\{H_j^{P_2}\}_{j=1}^J$.

\paragraph{The case of reconstructing $\sigma_a$ only.} We first consider the case where the absorption coefficient $\sigma_a$ is the only coefficient to be reconstructed. That is, the Gr\"uneisen coefficient and the scattering coefficient are both known. In this case, we can show that $\sigma_a$ can be uniquely recovered from only one initial pressure field. Moreover, the reconstruction of $\sigma_a$ is a relatively stable process.
\begin{theorem}\label{THM:Sigma-a}
	Under the assumptions in ($\cA$-i)-($\cA$-iv) and ~\eqref{EQ:Coeff Ass}, let $H_j^{P_2}$ and $\wt H_j^{P_2}$ be the initial pressure field corresponding to the coefficients $(\Xi, \sigma_a, \sigma_s)$ and $(\Xi, \wt\sigma_a, \sigma_s)$ respectively, induced by illumination source $f_j$. Assume further that $(\sigma_a, \wt\sigma_a)\in \cC^1(\bar\Omega)\times \cC^1(\bar\Omega)$, and $f_j$ is such that the corresponding solution to the simplified $P_2$ model satisfies the condition $(\phi_{1,j}-\frac{2}{3}\phi_{2,j})\neq 0$ a.e.. Then $H_j^{P_2}=\wt H_j^{P_2}$ a.e. implies $\sigma_a=\wt\sigma_a$. Moreover, we have the stability estimate
	\begin{equation}\label{EQ:Stab Sigma}
	\|(\sigma_a-\wt\sigma_a)(\phi_{1,j}-\dfrac{2}{3}\phi_{2,j})\|_{L^2(\Omega)} \le \mathfrak{c}\|H_j^{P_2}-\wt H_j^{P_2}\|_{L^2(\Omega)},
	\end{equation}
	where the constant $\mathfrak{c}$ depends on $\Omega$, $\Xi$, $\sigma_s$, $\underline{\alpha}$ and $\overline{\alpha}$.
\end{theorem}
\begin{proof}
	Let $\Phi_{i,j}=\phi_{i,j}-\wt\phi_{i,j}$, $i=1,2$, where $(\phi_{1,j}, \phi_{2,j})$ and $(\wt\phi_{1,j}, \wt\phi_{2,j})$ are solutions to~\eqref{EQ:SP2 simp} with $(\sigma_a, \sigma_s)$ and $(\wt\sigma_a, \sigma_s)$ respectively. We verify that $(\Phi_{1,j}, \Phi_{2,j})$ solves the following system
	\begin{equation}\label{EQ:SP2 simp Stab}
	\begin{array}{rcll}
	-\nabla\cdot D \nabla \Phi_{1,j} + \dfrac{H_j^{P_2}-\wt H_j^{P_2}}{\Xi} &=& 0, & \mbox{in}\ \Omega\\
	-\nabla\cdot D \nabla \Phi_{2,j} + \dfrac{5}{9\kappa'}\dfrac{1}{D}\Phi_{2,j} -\dfrac{2}{3\kappa} \dfrac{H_j^{P_2}-\wt H_j^{P_2}}{\Xi} &=& 0, & \mbox{in}\ \Omega\\
	\bn\cdot D \nabla \Phi_{1,j} + \dfrac{1}{2}\Phi_{1,j} - \dfrac{1}{8} \Phi_{2,j} & = & 0, & \mbox{on}\ \partial\Omega\\
	\bn\cdot D \nabla\Phi_{2,j} +\dfrac{7}{24\kappa}\Phi_{2,j}- \dfrac{1}{8\kappa}\Phi_{1,j} & = & 0, & \mbox{on}\ \partial\Omega
	\end{array}
	\end{equation}
	The coefficients (i.e. $D$ and $\sigma_{a2}$) in this system of equations are all known (since $\sigma_s$ is known), independent of the unknown absorption coefficient $\sigma_a$. With the regularity assumptions we have, standard elliptic theory~\cite{GiTr-Book00,McLean-Book00} shows that when $H_j^{P_2}=\wt H_j^{P_2}$ a.e., i.e. $H_j^{P_2}-\wt H_j^{P_2}=0$ a.e., the solution $(\Phi_{1,j}, \Phi_{2,j})=(0, 0)$ a.e.. This immediately implies that $(\phi_{1,j}, \phi_{2,j})=(\wt \phi_{1,j}, \wt\phi_{2,j})$. Therefore, $\sigma_a=\frac{H_j^{P_2}}{\phi_{1,j}-\frac{2}{3}\phi_{2,j}}=\frac{\wt H_j^{P_2}}{\wt\phi_{1,j}-\frac{2}{3}\wt\phi_{2,j}}=\wt\sigma_a$ if $\phi_{1,j}-\frac{2}{3}\phi_{2,j}\neq 0$ a.e.. This proves the uniqueness part of the theorem.
			
	To derive the stability estimate~\eqref{EQ:Stab Sigma}, we observe that 
	\[
	|\dfrac{H_j^{P_2}-\widetilde H_j^{P_2}}{\Xi}|=|\sigma_a(\phi_{1,j}-\dfrac{2}{3}\phi_{2,j})-\wt\sigma_a(\wt\phi_{1,j}-\dfrac{2}{3}\wt\phi_{2,j})|  = |(\sigma_a-\wt\sigma_a)(\phi_{1,j}-\dfrac{2}{3}\phi_{2,j})+\wt\sigma_a (\Phi_{1,j}-\dfrac{2}{3}\Phi_{2,j})|.
	\]
	This, together with the triangle inequality, gives that 
	\begin{equation}\label{EQ:Stab Sigma a}
	\|(\sigma_a-\wt\sigma_a)(\phi_{1,j}-\dfrac{2}{3}\phi_{2,j})\|_{L^2(\Omega)} \le \mathfrak{c}_1 |\dfrac{H_j^{P_2}-\widetilde H_j^{P_2}}{\Xi}\|_{L^2(\Omega)}+\mathfrak{c}_2\|(\Phi_j, \Psi_j)\|_{[L^2(\Omega)]^2},
	\end{equation}
	using the fact that $\sigma_a$ and $\wt\sigma_a$ are bounded as in the assumption ($\cA$-ii).
	
	On the other hand, the system~\eqref{EQ:SP2 simp Rec} provides us with the following stability estimate for $(\Phi_{1,j}, \Phi_{2,j})$:
	\begin{equation}\label{EQ:Stab Sigma b}
	\|(\Phi_{1,j}, \Phi_{2,j})\|_{[\cH^2(\Omega)]^2} \le \mathfrak{c}_3\|H_j^{P_2}-\wt H_j^{P_2}\|_{L^2(\Omega)}.
	\end{equation}
	The estimate~\eqref{EQ:Stab Sigma} then follows by combining~\eqref{EQ:Stab Sigma a} and~\eqref{EQ:Stab Sigma b}.
\end{proof}

The above proof is constructive in the sense that it provides an explicit procedure for the reconstruction of $\sigma_a$. To do that, we first solve
\begin{equation}\label{EQ:SP2 simp Rec}
\begin{array}{rcll}
-\nabla\cdot D \nabla \phi_{1,j} + \dfrac{H_j^{P_2}}{\Xi} &=& 0, & \mbox{in}\ \Omega\\
-\nabla\cdot D \nabla \phi_{2,j} + \dfrac{5}{9\kappa'}\dfrac{1}{D} \phi_{2,j} -\dfrac{2}{3\kappa} \dfrac{H_j^{P_2}}{\Xi} &=& 0, & \mbox{in}\ \Omega\\
\bn\cdot D \nabla \phi_{1,j} + \dfrac{1}{2}\phi_{1,j} - \dfrac{1}{8} \phi_{2,j} & = & \dfrac{1}{2} f_j(\bx), & \mbox{on}\ \partial\Omega\\
\bn\cdot D \nabla\phi_{2,j} +\dfrac{7}{24\kappa}\phi_{2,j} - \dfrac{1}{8\kappa}\phi_{1,j} & = & -\dfrac{1}{8\kappa} f_j(\bx), & \mbox{on}\ \partial\Omega
\end{array}
\end{equation}
for $(\phi_{1,j}, \phi_{2,j})$. We then simply reconstruct the absorption coefficient as $\sigma_a=\dfrac{H_j^{P_2}}{\Xi(\phi_{1,j}-\frac{2}{3}\phi_{2,j})}$ at points where $\phi_{1,j}(\bx)-\frac{2}{3}\phi_{2,j}(\bx)\neq 0$. When $J$ data sets are available, we reconstruct $\sigma_a$ as
\[
\sigma_a=\dfrac{\sum_{j=1}^J H_j^{P_2}}{\Xi\sum_{j=1}^J(\phi_{1,j}-\frac{2}{3}\phi_{2,j})}.
\]
Therefore, to reconstruct $\sigma_a$ from $J$ initial pressure fields, we only need to solve $J$ diffusion systems~\eqref{EQ:SP2 simp Rec} and perform some algebraic operations afterwards, even though the reconstruction problem is a nonlinear inverse problem.

The case of reconstructing the scattering coefficient or more than one coefficients are significantly more complicated, as demonstrated in the case of the classical diffusion model studied in~\cite{BaRe-IP11,BaUh-IP10}. We do not have results for these cases in the full nonlinear setting. We will instead study the problem in the linearized setting. 

We now use the second form of the simplified $P_2$ system given in~\eqref{EQ:SP2 simp V2}. We linearize the system formally following the differentiability result in Lemma~\ref{LMMA:Derivative}. We use  $(\Xi, \sigma_a, \sigma_s)$ and $(\delta \Xi, \delta\sigma_a, \delta\sigma_s)$ (note the equivalence $\delta D=-\delta \sigma_s/[3(1-g)\sigma_s^2]$) to denote respectively the background coefficients and the perturbation to the coefficients. We use 
\begin{equation}\label{EQ:Notations}
	(w_j, \phi_{2,j})=(\phi_{1,j}-\frac{2}{3}\phi_{2,j}, \phi_{2,j}), \qquad \mbox{and} \qquad 
	(\delta w_j, \delta\phi_{2,j})=(\delta\phi_{1,j}-\frac{2}{3}\delta\phi_{2,j}, \delta\phi_{2,j})
\end{equation} 
to denote the solutions to the background problem and the perturbations to the background solution caused by the perturbation in the coefficients, respectively. We then have that $(w_j, \phi_{2,j})$ solves
\begin{equation}\label{EQ:SP2 simp V2 Background}
\begin{array}{rcll}
-\nabla\cdot D \nabla w_j +(1+\frac{4}{9\kappa})\sigma_a w_j -\frac{10}{27\kappa' D}\phi_{2,j} &=& 0, & \mbox{in}\ \Omega\\
-\nabla\cdot D\nabla \phi_{2,j}+ \frac{5}{9\kappa'D} \phi_{2,j} -\frac{2}{3\kappa} \sigma_a w_j &=& 0, & \mbox{in}\ \Omega\\
\bn\cdot D \nabla w_j + \frac{6\kappa+1}{12} w_j+ \frac{15\kappa-10}{72\kappa} \phi_{2,j} & = &  \frac{6\kappa+1}{12} f_j(\bx), & \mbox{on}\ \partial\Omega\\
\bn\cdot D \nabla \phi_{2,j} +\frac{5}{24\kappa} \phi_{2,j} - \frac{1}{8\kappa} w_j & = &-\frac{1}{8\kappa}f_j(\bx), & \mbox{on}\ \partial\Omega
\end{array}
\end{equation}
while  $(\delta w_j, \delta \phi_{2,j})$ solves
\begin{equation}\label{EQ:SP2 simp V2 Lin}
\begin{array}{rcll}
-\nabla\cdot D \nabla \delta w_j +(1+\frac{4}{9\kappa})\sigma_a \delta w_j -\frac{10}{27\kappa' D}\delta\phi_{2,j} &=& \nabla\cdot \delta D \nabla w_j -(1+\frac{4}{9\kappa}) w_j \delta\sigma_a -\frac{10}{27\kappa' D^2}\phi_{2,j}\delta D, & \mbox{in}\ \Omega\\
-\nabla\cdot D\nabla \delta\phi_{2,j}+ \frac{5}{9\kappa'D} \delta \phi_{2,j} -\frac{2\sigma_a}{3\kappa} \delta w_j &=& \nabla\cdot\delta D\nabla \phi_{2,j}+\frac{5\phi_{2,j}}{9\kappa'D^2} \delta D +\frac{2 w_j}{3\kappa}\delta\sigma_a, & \mbox{in}\ \Omega\\
\bn\cdot D \nabla \delta w_j + \frac{6\kappa+1}{12} \delta w_j+ \frac{15\kappa-10}{72\kappa} \delta\phi_{2,j} & = & 0, & \mbox{on}\ \partial\Omega\\
\bn\cdot D \nabla \delta \phi_{2,j} +\frac{5}{24\kappa} \delta \phi_{2,j} - \frac{1}{8\kappa} \delta w_j & =& 0, & \mbox{on}\ \partial\Omega
\end{array}
\end{equation}
where we have used the fact that $D_{|\partial\Omega}$ is known (since ${\sigma_s}_{|\partial\Omega}$ is known) in the boundary conditions. 

The perturbed initial pressure field, the data, now take the form:
\begin{equation}\label{EQ:Int Data SP2 Pert}
\delta H_j^{P_2}=\left(\delta \Xi\sigma_a + \Xi\delta \sigma_a \right)w_j+\Xi \sigma_a\delta w_j.
\end{equation}
The linearized data show that if the background Gr\"uneisen coefficient $\Xi=0$, then $\delta H_j^{P_2}=\delta\Xi  \sigma_a(\phi_{1,j}-\frac{2}{3}\phi_{2,j})$. Therefore, we can reconstruct $\delta\Xi$, but not the other parameters. If the background absorption coefficient $\sigma_a=0$ (which means that the medium is weakly absorbing), then $\delta H_j^{P_2}=\Xi \delta\sigma (\phi_{1,j}-\frac{2}{3}\phi_{2,j})$. Therefore, we can reconstruct $\delta\sigma_a$ but not the other parameters.

\paragraph{The case of reconstructing $\sigma_s$ only.} We start with the case of reconstructing only the scattering coefficient. 

We have the following result in the linearized setting.
\begin{theorem}\label{THM:Recon Sigma-s}
	Under the assumptions in ($\cA$-i)-($\cA$-iv) and~\eqref{EQ:Coeff Ass} for the domain and the background coefficients, let $\delta H_j^{P_2}$ and $\delta \wt H_j^{P_2}$ be two perturbed data sets generated with perturbed coefficients $\delta\sigma_s$ and $\delta\wt\sigma_s$ respectively. Then we have the following bound on the reconstruction:
	\begin{equation}\label{EQ:Stab Sigma-s}
	\int_\Omega Q_j\left(\dfrac{\delta D}{D}-\dfrac{\delta\wt{D}}{D}\right)^2 d\bx \le \mathfrak{c}\|\delta H_j^{P_2}-\delta \wt H_j^{P_2}\|_{\cH^2(\Omega)}
	\end{equation}
	where the constant $\mathfrak{c}$ depends on $\Omega$ and the background coefficients, and $Q_j$ is defined as
	\begin{equation}\label{EQ:Uniq Ass}
	Q_j(\bx) := ((1+\frac{4}{9\kappa})\sigma_a w_j-4\gamma\phi_{2,j})\phi_{2,j} -\bu\cdot\nabla \phi_{2,j} -\dfrac{1}{3}D|\nabla\phi_{2,j}|^2.
	\end{equation}
If $Q_j(\bx)> -\frac{10}{27\kappa' D}\phi_{2,j}^2$ or $Q_j(\bx)\ge -\frac{10}{27\kappa' D}\phi_{2,j}^2$ and $|\nabla w_j|\ge \eps>0$ for some $\eps$, then $\delta H_j^{P_2}=\delta \wt H_j^{P_2}$ a.e on $\bar\Omega$ implies $\delta\sigma_s=\delta\wt{\sigma}_s$.
\end{theorem}
\begin{proof}
	When only $\sigma_s$ is sought, the perturbed datum~\eqref{EQ:Int Data SP2 Pert} simplifies to
	\begin{equation}\label{EQ:Int Data SP2 Pert sigma-s}
	\dfrac{\delta H_j^{P_2}}{\Xi}=\sigma_a \delta w_j,
	\end{equation}
	while the perturbed simplified $P_2$ system~\eqref{EQ:SP2 simp V2 Lin} simplifies to
	\begin{equation}\label{EQ:SP2 simp V2 Lin Sigma-s}
	\begin{array}{rcll}
	-\nabla\cdot D \nabla \delta w_j +(1+\frac{4}{9\kappa})\sigma_a \delta w_j -\frac{10}{27\kappa'D} \delta \phi_{2,j} &=& \nabla\cdot \delta D \nabla w_j-\frac{10\phi_{2,j}}{27\kappa'D^2} \delta D, & \mbox{in}\ \Omega\\
	-\nabla\cdot D\nabla \delta\phi_{2,j}+ \frac{5}{9\kappa'D} \delta \phi_{2,j} -\frac{2}{3\kappa} \sigma_a\delta w_j &=& \nabla\cdot\delta D\nabla \phi_{2,j}+\frac{5\phi_{2,j}}{9\kappa'D^2} \delta D, & \mbox{in}\ \Omega\\ \bn\cdot D \nabla \delta w_j + \frac{6\kappa+1}{12\kappa}\delta w_j+\frac{15\kappa-10}{72\kappa}\delta \phi_{2,j} & = & 0, & \mbox{on}\ \partial\Omega\\ \bn\cdot D \nabla\delta\phi_{2,j} +\frac{5}{24\kappa}\delta\phi_{2,j} - \frac{1}{8\kappa}\delta w_j & = &0, & \mbox{on}\ \partial\Omega
	\end{array}
	\end{equation}
	
	Using~\eqref{EQ:Int Data SP2 Pert sigma-s}, we can further rewrite~\eqref{EQ:SP2 simp V2 Lin Sigma-s} into
	\begin{equation}\label{EQ:SP2 simp V2 Lin Sigma-s-2}
	\begin{array}{rcll}
	-\frac{10}{27\kappa'D} \delta \phi_{2,j} &=& \nabla\cdot \delta D \nabla w_j-\frac{10\phi_{2,j}}{27\kappa'D^2} \delta D \\
	& & \qquad \qquad + \nabla\cdot D \nabla \frac{\delta H_j^{P_2}}{\Xi\sigma_a} -(1+\frac{4}{9\kappa})\frac{\delta H_j^{P_2}}{\Xi}, & \mbox{in}\ \Omega\\
	-\nabla\cdot D\nabla \delta\phi_{2,j}+ \frac{5}{9\kappa'D} \delta \phi_{2,j} &=& \nabla\cdot\delta D\nabla \phi_{2,j}+\frac{5\phi_{2,j}}{9\kappa'D^2} \delta D+\frac{2}{3\kappa} \frac{\delta H_j^{P_2}}{\Xi}, & \mbox{in}\ \Omega\\ 
\bn\cdot D \nabla\delta\phi_{2,j} +\frac{5}{24\kappa}\delta\phi_{2,j}  & = &\frac{1}{8\kappa}\frac{\delta H_j^{P_2}}{\Xi\sigma_a}, & \mbox{on}\ \partial\Omega\\
	\delta D & = & 0, & \mbox{on}\ \partial\Omega
	\end{array}
	\end{equation}
	Note that we have added the natural boundary condition for $\delta D$ that we assumed in the assumption ($\cA$-iv). This can now be seen as a system of partial differential equations with $(\delta\phi_{2,j}, \delta D)$ as the unknown.

	Let $\bu=D\nabla w_j$, $\mu=\delta D/D$, $\gamma=\frac{10}{27\kappa' D}$ and $Y_j=\nabla\cdot D \nabla \frac{\delta H_j^{P_2}}{\Xi\sigma_a} -(1+\frac{4}{9\kappa})\frac{\delta H_j^{P_2}}{\Xi}$. We then verify that the first equation in~\eqref{EQ:SP2 simp V2 Lin Sigma-s-2} can be written as
	\begin{equation}\label{EQ:SP2 simp V2 Lin Sigma-s-4}
		\nabla\cdot\mu\bu-\gamma\phi_{2,j}\mu+\gamma\delta\phi_{2,j}+Y_j=0,
	\end{equation}
	and the first equation in the background system~\eqref{EQ:SP2 simp V2 Background} can be written as
	\begin{equation}\label{EQ:SP2 simp V2 Background-2}
		-\nabla\cdot\bu+(1+\frac{4}{9\kappa})\sigma_a w_j-\gamma\phi_{2,j}=0.
	\end{equation}
	Moreover, we observe that for any scalar function $\gamma$ and vector function $\bu$, we have
	\begin{equation}\label{EQ:Renorm-1}
		\nabla\cdot \mu^2 \bu-2\mu\nabla\cdot\mu \bu+\mu^2\nabla\cdot \bu=0.
	\end{equation}
	Using~\eqref{EQ:SP2 simp V2 Lin Sigma-s-4} and ~\eqref{EQ:SP2 simp V2 Background-2}, we can write~\eqref{EQ:Renorm-1} as
	\begin{equation}\label{EQ:Renorm-2}
		\nabla\cdot \mu^2 \bu-\mu^2(3\gamma\phi_{2,j}-(1+\frac{4}{9\kappa})\sigma_a w_j)+2\mu (\gamma\delta\phi_{2,j}+Y_j)=0.
	\end{equation}
	We multiply this equation by $\phi_{2,j}$ and integrate over $\Omega$ to get
	\begin{equation}\label{EQ:Renorm-3}
		\int_\Omega \mu^2 \left[ -\bu\cdot\nabla \phi_{2,j} -(3\gamma\phi_{2,j}-(1+\frac{4}{9\kappa})\sigma_a w_j)\phi_{2,j}\right] d\bx  +2\int_\Omega \mu (\gamma \delta\phi_{2,j}+Y_j)\phi_{2,j} d\bx = 0.
	\end{equation}

	Meanwhile, we can multiply the second equation in~\eqref{EQ:SP2 simp V2 Lin Sigma-s-2} by $\delta\phi_{2,j}$ and integrate over $\Omega$ to get
	\begin{equation}\label{EQ:Renorm-4}
		\int_\Omega \left[ D|\nabla\delta\phi_{2,j}|^2+ \dfrac{3}{2} \gamma |\delta\phi_{2,j}|^2+\mu D\nabla\phi_{2,j}\cdot\nabla\delta\phi_{2,j} -\dfrac{3}{2} \gamma \mu \phi_{2,j}\delta\phi_{2,j} - \frac{2}{3\kappa} \frac{\delta H_j^{P_2}}{\Xi}\delta\phi_{2,j} \right] d\bx = 0.
	\end{equation}

	We can now combine~\eqref{EQ:Renorm-3} and~\eqref{EQ:Renorm-4} to get
	\begin{multline}\label{EQ:Renorm-5}
		\int_\Omega \mu^2 \left[ -\bu\cdot\nabla \phi_{2,j} -(3\gamma\phi_{2,j}-(1+\frac{4}{9\kappa})\sigma_a w_j)\phi_{2,j}\right] d\bx + \int_{\Omega} \left[\dfrac{4}{3} D|\nabla\delta\phi_{2,j}|^2 + 2 \gamma |\delta\phi_{2,j}|^2\right] d\bx \\
+\dfrac{4}{3} \int_\Omega \mu D\nabla\phi_{2,j}\cdot\nabla\delta\phi_{2,j} d\bx + \int_\Omega 2 \mu Y_j \phi_{2,j} d\bx -\int_\Omega \frac{8}{9\kappa} \frac{\delta H_j^{P_2}}{\Xi}\delta\phi_{2,j} = 0.
	\end{multline}
	Using the fact that $\forall x, y\in\bbR$, $xy\le \frac{1}{2}((ax)^2+(y/a)^2)$, $\forall a\neq 0$, we have the following bounds for the last three terms in above equation:
	\begin{equation*}
	\begin{array}{c}
		\dint_\Omega \mu D\nabla\phi_{2,j}\cdot\nabla\delta\phi_{2,j} d\bx \le  \int_\Omega \left[\dfrac{1}{4} \mu^2 D|\nabla\phi_{2,j}|^2 +D |\nabla\delta\phi_{2,j}|^2 \right] d\bx,\\
	\dint_\Omega 2 \mu Y_j \phi_{2,j} d\bx \le  \int_\Omega \left[\gamma \mu^2 \phi_{2,j}^2 +\dfrac{Y_j^2 }{\gamma} \right] d\bx, \quad 
	\int_\Omega \frac{8\delta H_j^{P_2}}{9\kappa\Xi} \delta\phi_{2,j} \le  \int_\Omega \left[(\frac{2\sqrt{2} \delta H_j^{P_2}}{9\kappa\Xi\sqrt{\gamma}})^2 +2\gamma \delta\phi_{2,j}^2 \right] d\bx.
	\end{array} 
	\end{equation*}
	These bounds can be combined with ~\eqref{EQ:Renorm-5} to conclude that
	\begin{equation}\label{EQ:Renorm-7}
		\int_\Omega  \mu^2 \left[((1+\frac{4}{9\kappa})\sigma_a w_j-4\gamma\phi_{2,j})\phi_{2,j} -\bu\cdot\nabla \phi_{2,j} -\dfrac{1}{3}D|\nabla\phi_{2,j}|^2 \right]  d\bx \le \mathfrak{c} \|\delta H_j^{P_2}\|^2_{\cH^2(\Omega)},
	\end{equation}
	with $\mathfrak{c}$ depending on the background coefficients as well as $\Omega$. The stability result in~\eqref{EQ:Stab Sigma-s} then follows from the linearity of the problem.

	To prove the uniqueness claim, we observe that when $\delta H_j^{P_2}=0$ in ~\eqref{EQ:Renorm-5}, ~\eqref{EQ:Renorm-7} becomes
	\begin{equation}\label{EQ:Renorm-8}
		\int_\Omega \left\{\mu^2 \left[Q_j+\gamma\phi_{2,j}^2 \right] + \gamma |\delta\phi_{2,j}|^2\right\} d\bx \le 0.
	\end{equation}
	When $Q_j>-\gamma\phi_{2,j}^2$, we conclude that $\mu\equiv 0\equiv \delta\phi_{2,j}$ from the above inequality. When $Q_j+\gamma\phi_{2,j}^2=0$, we conclude from the above inequality that $\delta\phi_{2,j}\equiv 0$. The first equation in~\eqref{EQ:SP2 simp V2 Lin Sigma-s-2} then simplifies to, with $(\delta\phi_{2,j}, \delta H_j^{P_2})=(0,0)$,
	\begin{equation*}
		\nabla\cdot \mu \bu-\gamma\phi_{2,j}\mu=0,\ \ \mbox{in}\ \Omega, \qquad \mu=0, \ \ \mbox{on}\ \partial\Omega.
	\end{equation*}
	This transport equation admits the unique solution $\mu=0$ when $|\bu|\ge \eps>0$ for some $\eps$~\cite{BaRe-IP11,DiLi-AM89}. The proof is complete.

\end{proof}

We now consider the case where more than one coefficient is to be reconstructed. We focus on the practically important cases of reconstructing $(\delta \Xi, \delta\sigma_a)$ and $(\delta\sigma_a, \delta\sigma_s)$.

\paragraph{The case of reconstructing $(\delta\Xi, \delta\sigma_a)$.} In this case, the scattering coefficient $\sigma_s$ (and therefore $D$) is known. Therefore the linearized simplified $P_2$ equation~\eqref{EQ:SP2 simp V2 Lin} reduces to:
\begin{equation}\label{EQ:SP2 simp V2 Lin XS}
\begin{array}{rcll}
-\nabla\cdot D \nabla \delta w_j +(1+\frac{4}{9\kappa})\sigma_a \delta w_j-\frac{10}{27\kappa' D}\delta\phi_{2,j} &=& -(1+\frac{4}{9\kappa})\delta\sigma_a w_j, & \mbox{in}\ \Omega\\
-\nabla\cdot D\nabla \delta\phi_{2,j}+ \frac{5}{9\kappa'}\frac{1}{D} \delta \phi_{2,j} -\frac{2}{3\kappa}\sigma_a \delta w_j &=& \frac{2}{3\kappa}\delta\sigma_a w_j, & \mbox{in}\ \Omega\\
\bn\cdot D \nabla \delta w_j + \frac{6\kappa+1}{12}\delta w_j+\frac{15\kappa-10}{72\kappa}\delta\phi_{2,j} & = & 0, & \mbox{on}\ \partial\Omega\\
\bn\cdot D \nabla\delta\phi_{2,j} +\frac{5}{24\kappa}\delta\phi_{2,j} - \frac{1}{8\kappa}\delta w_j & = &0, & \mbox{on}\ \partial\Omega
\end{array}
\end{equation}
Since the linearized data~\eqref{EQ:Int Data SP2 Pert} does not depend on the scattering coefficient $\sigma_s$ explicitly, it remains in the original form in this case.

We now develop a two-stage procedure for the reconstruction of $(\delta \Xi, \delta \sigma_a)$. We first eliminate $\delta \Xi$ from the system to reconstruct $\delta\sigma_a$. To do that, we check that, for any $i\neq j$,
\begin{equation}\label{EQ:Int Data SP2 Pert XS Ratio}
\delta H_{ij}^{P_2} \equiv w_j \dfrac{\delta H_{i}^{P_2}}{\Xi\sigma_a}-w_i \dfrac{\delta H_{j}^{P_2}}{\Xi\sigma_a}=w_j\delta w_i-w_i\delta w_j,
\end{equation}
We then observe that $\delta H_{ij}^{P_2}$ does NOT depend explicitly on the coefficient perturbation $\delta \Xi$. Moreover, the equations for the perturbations in~\eqref{EQ:SP2 simp V2 Lin XS} depend only on $\delta\sigma_a$. We could hope to reconstruct $\delta\sigma_a$ out of ~\eqref{EQ:SP2 simp V2 Lin XS} and~\eqref{EQ:Int Data SP2 Pert XS Ratio}.

We have the following partial result on the reconstruction of $\delta\sigma_a$.
\begin{theorem}\label{LMM:W 2 sigma-a}
	Under the assumptions in ($\cA$-i)-($\cA$-iv) and ~\eqref{EQ:Coeff Ass}, let $\delta w_j$ and $\wt {\delta w_j}$ be solutions to~\eqref{EQ:SP2 simp V2 Lin XS} with coefficients $\delta \sigma_a$ and $\wt{\delta\sigma_a}$ respectively. Assume further that the illumination source $f_j$ is selected such that the background solution $w_j\neq 0$. Then $\delta w_j=\wt {\delta w_j}$ a.e. implies that $\delta \sigma_a=\wt{\delta\sigma_a}$. Moreover, we have the following stability bound:
	\begin{equation}\label{EQ:Stab XS Sigma-a}
	\mathfrak{c}\|\delta w_j- \wt {\delta w_j}\|_{\cH^2(\Omega)}\le \|(\delta\sigma_a-\wt{\delta\sigma}_a)w_j\|_{L^2(\Omega)}\le \tilde{\mathfrak{c}}\|\delta w_j- \wt {\delta w_j}\|_{\cH^2(\Omega)},
	\end{equation}
	where $\mathfrak{c}$ and $\tilde{\mathfrak{c}}$ are constants that depend on the domain $\Omega$, the background coefficients and the background solution $(w_j, \phi_{2,j})$.
\end{theorem}
\begin{proof}
	We first prove the injectivity claim. Let $\delta w_j=0$. Then the first equation in~\eqref{EQ:SP2 simp V2 Lin XS} implies that
	\[
	\delta\sigma_a w_j =\frac{10}{27\kappa' D}\delta\phi_{2,j}/(1+\frac{4}{9\kappa}).
	\]
	The second equation in~\eqref{EQ:SP2 simp V2 Lin XS}, together with its boundary condition, then simplifies to
	\begin{equation*}
	\begin{array}{rcll}
	-\nabla\cdot D\nabla \delta\phi_{2,j}+ \kappa'''\delta \phi_{2,j}  &=& 0, & \mbox{in}\ \Omega\\
	\bn\cdot D \nabla\delta\phi_{2,j} +\frac{5}{24\kappa}\delta\phi_{2,j} & = &0, & \mbox{on}\ \partial\Omega
	\end{array}
	\end{equation*}
	where $k'''=\frac{5}{9\kappa'}\frac{1}{D}-\frac{2}{3\kappa}\frac{10}{27\kappa' D}/(1+\frac{4}{9\kappa}) )>0$. This equation admits only the trivial solution $\delta \phi_{2,j}\equiv 0$. Therefore $\delta\sigma_a\equiv 0$.
	
	To derive the stability bound~\eqref{EQ:Stab XS Sigma-a}, we first observe that the left inequality follows directly from classical theory for elliptic systems~\cite{GiTr-Book00, McLean-Book00}. To obtain the right inequality, we use the first equation in~\eqref{EQ:SP2 simp V2 Lin XS}. We take the square of both sides of the equation, integrate over $\Omega$, and use the triangle and the H\"older's inequalities to obtain
	\begin{equation}\label{EQ:Stab XS-1}
		\|\delta\sigma_a w_j\|_{L^2(\Omega)}^2 \le \mathfrak{c}_1 \left(\|\delta w_j\|_{\cH^2}^2 + \|\delta\phi_{2,j}\|_{L^2(\Omega)}^2 + \|\delta w_j\|_{\cH^1(\Omega)}\|\delta\phi_{2,j}\|_{L^2(\Omega)}\right).
	\end{equation}
	We now multiply the first equation in~\eqref{EQ:SP2 simp V2 Lin XS} by $\frac{2}{3\kappa}$, the second equation by $1+\frac{4}{9\kappa}$, and add the results together to get, after eliminating the factor $1+\frac{4}{9\kappa}$,
	\begin{equation}\label{EQ:WP}
		  -\nabla\cdot D\nabla \delta\phi_{2,j}+\frac{5\kappa}{(4+9\kappa)\kappa' D}\delta\phi_{2,j}=\frac{6}{4+9\kappa} \nabla\cdot D \nabla \delta w_j, \ \ \mbox{in}\ \ \Omega.
	\end{equation}
	Moreover, from the boundary condition for $\delta \phi_{2,j}$ in~\eqref{EQ:SP2 simp V2 Lin XS} we have
	\begin{equation}\label{EQ:WP BC}
	\bn\cdot D \nabla\delta\phi_{2,j} +\frac{5}{24\kappa}\delta\phi_{2,j} = \frac{1}{8\kappa}\delta w_j, \ \ \mbox{on}\ \ \partial\Omega
	\end{equation}
	We can therefore look at~\eqref{EQ:WP} and~\eqref{EQ:WP BC} as an elliptic equation for $\delta\phi_{2,j}$ and conclude from classical theory~\cite{GiTr-Book00,Evans-Book10} that 
	\begin{equation}\label{EQ:Stab XS-2}
		\|\delta \phi_{2,j}\|_{\cH^2(\Omega)}\le \mathfrak{c}_2\|\nabla\cdot D\nabla \delta w_j\|_{L^2(\Omega)}+\|\delta w_j\|_{L^2(\partial\Omega)}\le \tilde{\mathfrak{c}}\|\delta w_j\|_{\cH^2(\Omega)}.
	\end{equation}
	We now combine~\eqref{EQ:Stab XS-1} and ~\eqref{EQ:Stab XS-2} to get the right equality in~\eqref{EQ:Stab XS Sigma-a}.	
\end{proof}

So long as we could select two illumination sources $f_i$ and $f_j$ such that the background densities $w_i$ and $w_j$ do not destroy the invertibility of the map $\delta\sigma_a\mapsto \delta H_{ij}^{P_2}= w_i \delta w_j - w_j \delta w_i$, note again that both $\delta\sigma_a\mapsto \delta w_i$ and $\delta\sigma_a\mapsto \delta w_j$ are invertible by the previous theorem, we could uniquely reconstruct $\delta\sigma_a$ from $\delta H_{ij}^{P_2}$. 

To perform numerical reconstruction of $\delta\sigma_a$ from $J$ data sets, we use the usual least-square inversion method. We minimize the functional
\begin{equation}\label{EQ:OBJ XS}
	\Psi(\delta\sigma_a)=\sum_{1\le i< j\le J}\|w_i\delta w_j-w_j\delta w_i-\delta H_{ij}^{P_2*}\|_{L_2(\Omega)}^2+\beta\|\nabla \delta\sigma_a\|_{[L_2(\Omega)]^3}^2.
\end{equation}
Note that here we form the difference data using all $(i, j)$ pairs satisfying $i<j$. There are totally $J(J-1)/2$ such pairs. We solve this minimization problem using the SNOPT algorithm we described in the previous section, even though this problem is linear. Once we reconstructed $\delta\sigma_a$, we can reconstruct $\delta\Xi$ using the data~\eqref{EQ:Int Data SP2 Pert}:
\[
	\delta\Xi=\dfrac{\sum_{j=1}^J \delta H_j^{P_2}-\Xi\delta \sigma_a\sum_{j=1}^Jw_j-\Xi\sigma_a\sum_{j=1}^J\delta w_j}{\sigma_a \sum_{j=1}^J w_j}.
\]

\paragraph{The case of reconstructing $(\delta\sigma_a, \delta\sigma_s)$.} In the case where $\Xi$ is assumed known, the perturbed data~\eqref{EQ:Int Data SP2 Pert} simplify to
\begin{equation}\label{EQ:Int Data SP2 Pert SS}
	\dfrac{\delta H_j^{P_2}}{\Xi} = \sigma_a \delta w_j + \delta \sigma_a w_j.
\end{equation}
This simplification allows us to rewrite the system of equations for the perturbations, that is, system~\eqref{EQ:SP2 simp V2 Lin}, into the form
\begin{equation}\label{EQ:SP2 simp V2 Lin SS}
\begin{array}{rcll}
-\nabla\cdot D \nabla \delta w_j -\frac{10}{27\kappa'D} \delta \phi_{2,j} &=& \nabla\cdot \delta D \nabla w_j-\frac{10\phi_{2,j}}{27\kappa'D^2} \delta D -\frac{9\kappa+4}{9\kappa} \frac{\delta H_j^{P_2}}{\Xi}, & \mbox{in}\ \Omega\\
-\nabla\cdot D\nabla \delta\phi_{2,j}+ \frac{5}{9\kappa'D} \delta \phi_{2,j} &=& \nabla\cdot\delta D\nabla \phi_{2,j}+\frac{5\phi_{2,j}}{9\kappa'D^2} \delta D +\frac{2 }{3\kappa}\frac{\delta H_j^{P_2}}{\Xi}, & \mbox{in}\ \Omega\\
\bn\cdot D \nabla \delta w_j + \frac{6\kappa+1}{12}\delta w_j +\frac{15\kappa-10}{72\kappa} \delta\phi_{2,j}& = &0, & \mbox{on}\ \partial\Omega\\
\bn\cdot D \nabla\delta\phi_{2,j} +\frac{5}{24\kappa}\delta\phi_{2,j} -\frac{1}{8\kappa}\delta w_j & = & 0, & \mbox{on}\ \partial\Omega
\end{array}
\end{equation}
This system does \emph{not} depend explicitly on $\delta\sigma_a$. 

The simplification~\eqref{EQ:Int Data SP2 Pert SS} also allows us to form the difference data $\delta H_{ij}^{P_2}$ in the same way as in~\eqref{EQ:Int Data SP2 Pert XS Ratio}. The difference data $\delta H_{ij}^{P_2}$ in~\eqref{EQ:Int Data SP2 Pert XS Ratio} do not depend on $\delta\sigma_a$ either.

Let us again consider a two-stage procedure for the reconstruction of $(\delta\sigma_a, \delta\sigma_s)$. We first use the combination of~\eqref{EQ:SP2 simp V2 Lin SS} and ~\eqref{EQ:Int Data SP2 Pert XS Ratio} to reconstruct the perturbation of the scattering coefficient, $\delta\sigma_s$ (or equivalently $\delta D$). We then reconstruct $\delta\sigma_a$ once $\delta D$ has been reconstructed.

The following result is a simple corollary of Theorem~\ref{THM:Recon Sigma-s}.
\begin{corollary}\label{LMM:W 2 sigma-s}
	Under the same assumptions made in Theorem~\ref{THM:Recon Sigma-s}, the linear map 
	\[
		(\delta H_j^p, \delta w_j): 
		\begin{array}{c}
			\delta \sigma_s \mapsto (\delta H_j^p, \delta w_j)\\
			L^2(\Omega)\mapsto \cH^2(\Omega) \times \cH^2(\Omega)
		\end{array}
	\]
	is injective when $Q_j(\bx)> -\frac{10}{27\kappa' D}\phi_{2,j}^2$ or $Q_j(\bx)\ge -\frac{10}{27\kappa' D}\phi_{2,j}^2$ and $|\nabla w_j|\ge \eps>0$ for some $\eps$. Moreover, we have the following stability bound:
	\begin{equation}\label{EQ:Stab SS Sigma-s}
		\int_\Omega (\delta\sigma_s-\wt{\delta\sigma}_s)^2 Q_j d\bx \le \mathfrak{c}\left(\|\delta H_j^{P_2}-\delta \wt H_j^{P_2}\|_{L^2(\Omega)}^2+\|\delta w_j-\delta \wt w_j\|_{\cH^2(\Omega)}^2\right),
	\end{equation}
	where $\mathfrak{c}$ is a constant that depends on $\Omega$ and the background coefficients.
\end{corollary}
\begin{proof}
	The proof is almost identical to that of Theorem~\ref{THM:Recon Sigma-s}. If we move the terms involving $\delta w_j$ to the right hand side, the system~\eqref{EQ:SP2 simp V2 Lin SS} has exactly the same structure as~\eqref{EQ:SP2 simp V2 Lin Sigma-s}. The stability bound~\eqref{EQ:Stab SS Sigma-s} follows from the same argument for~\eqref{EQ:Stab Sigma-s}. The injectivity claim follows when taking $(\delta H_j^{P_2}, \delta w_j)=(0, 0)$. 
\end{proof}

To reconstruct $\delta D$ from data $(\delta H_i^{P_2}, \delta H_{ij}^{P_2})$ (or equivalently $(\delta H_i^{P_2}, \delta H_j^{P_2})$ ), we need to select two illumination sources $f_i$ and $f_j$ such that the background densities $w_i$ and $w_j$ do not destroy the injectivity of the map $\delta D \mapsto (\delta H_i^{P_2}, \delta H_{ij}^{P_2})= (\delta H_i^{P_2}, w_i \delta w_j - w_j \delta w_i)$. Computationally, we solve the reconstruction problem by solving a least-square minimization problem with the same objective function in~\eqref{EQ:OBJ XS} (besides the regularization term which is now on $\delta D$). Once we reconstructed $\delta\sigma_a$, we can reconstruct $\delta\sigma_a$ using the data~\eqref{EQ:Int Data SP2 Pert SS}:
\[
\delta\sigma_a=\dfrac{\sum_{j=1}^J \delta H_j^{P_2}/\Xi-\sigma_a\sum_{j=1}^J\delta w_j}{\sum_{j=1}^J w_j}.
\]

\paragraph{Comparing simplified $P_2$ and $P_1$ reconstructions.} The main motivation for using more accurate forward light propagation models in PAT is that the reconstructions based these  models are more accurate. For instance, the difference between the reconstruction of the Gr\"uneisen coefficient from the radiative transfer model~\eqref{EQ:RTE} and that from the classical diffusion model~\eqref{EQ:Diff}, using the same data $H$, is given as
\[
	\Xi_{rte}-\Xi_{diff} = \dfrac{H(\varphi-\int_{\bbS^2}u(\bx,\bv) d\bv)}{\sigma_a (\int_{\bbS^2}u(\bx,\bv) d\bv) \varphi(\bx)}
\]
In this simple case, the error in the reconstruction of $\Xi$ is proportional to the difference between the solutions to the two models. In the next theorem, we characterize the difference between the reconstruction of $\sigma_a$ using the simplified $P_2$ model~\eqref{EQ:SP2 simp} and that using the classical diffusion model~\eqref{EQ:Diff}.
\begin{theorem}\label{THM:Stab Mult Model}
	Let $\Omega$, $c$ and $\Xi$ satisfy the assumptions in ($\cA$-i)-($\cA$-iv) and assume that the assumptions in~\eqref{EQ:Coeff Ass} holds. Let $\sigma_a^{P_2}$ and $\sigma_a^{P_1}$ be the absorption coefficients reconstructed with the simplified $P_2$ model~\eqref{EQ:SP2 simp} and the classical diffusion model~\eqref{EQ:Diff} respectively, using the same datum $H$. Assume that $H$ is also known on the boundary $\partial\Omega$.  Then we have 
	\begin{equation}\label{EQ:Rec Sigma-a Dif}
			\sigma_a^{P_2}-\sigma_a^{P_1} =\dfrac{H}{\Xi}\dfrac{\phi-(\phi_1-\frac{2}{3}\phi_2)}{\phi(\phi_1-\frac{2}{3}\phi_2)}=\dfrac{2}{3}\dfrac{\Xi}{H\sigma_a^{P_2} \sigma_a^{P_1}}\phi_2
	\end{equation}
	where $\phi_2$ solves
	\begin{equation}\label{EQ:SP2 simp phi2 Rec}
	\begin{array}{rcll}
	-\nabla\cdot D \nabla \phi_2 + \dfrac{5}{9\kappa'D}\phi_2  &=& \dfrac{2}{3\kappa} \dfrac{H}{\Xi}, & \mbox{in}\ \Omega\\
	\bn\cdot D \nabla\phi_2 +\dfrac{5}{24\kappa}\phi_2 & = & \dfrac{1}{8\kappa}(\dfrac{H}{\sigma_a \Xi}-f(\bx)), & \mbox{on}\ \partial\Omega
	\end{array}
	\end{equation}
\end{theorem}
\begin{proof}
	We first observe that $\sigma_a^{P_2}$ and $\sigma_a^{P_1}$ can be explicitly reconstructed with the following procedures:
	\begin{equation}\label{EQ:Diff Sigma}
	\sigma_a^{P_2}=\dfrac{H}{\Xi(\phi_1-\frac{2}{3}\phi_2)} \qquad \mbox{and} \qquad \sigma_a^{P_1}=\dfrac{H}{\Xi \phi}
	\end{equation}
	where $(\phi_1, \phi_2)$ is the solution to the simplified $P_2$ system~\eqref{EQ:SP2 simp} with $\sigma_a(\phi_1-\frac{2}{3}\phi_2)$ replaced with $H/\Xi$, and $\phi$ is the solution to the diffusion model~\eqref{EQ:Diff} with $\sigma_a \phi$ replaced by $H/\Xi$. In other words, $(\phi_1, \phi_2)$ solves
	\begin{equation}\label{EQ:SP2 simp Rec A}
	\begin{array}{rcll}
	-\nabla\cdot D \nabla \phi_1 &=& - \dfrac{H}{\Xi} , & \mbox{in}\ \Omega\\
	-\nabla\cdot D \nabla \phi_2 + \dfrac{5}{9\kappa'D}\phi_2  &=& \dfrac{2}{3\kappa} \dfrac{H}{\Xi}, & \mbox{in}\ \Omega\\
	\bn\cdot D \nabla \phi_1 + \dfrac{5}{16}\phi_1  & = & \dfrac{1}{2} f(\bx)-\dfrac{3}{16}\dfrac{H}{\sigma_a \Xi}, & \mbox{on}\ \partial\Omega\\
	\bn\cdot D \nabla\phi_2 +\dfrac{5}{24\kappa}\phi_2 & = & \dfrac{1}{8\kappa}(\dfrac{H}{\sigma_a \Xi}-f(\bx)), & \mbox{on}\ \partial\Omega
	\end{array}
	\end{equation}
	while $\phi$ solves
	\begin{equation}\label{EQ:Diff Rec}
	\begin{array}{rcll}
	-\nabla\cdot D \nabla \phi  &=& -\dfrac{H}{\Xi}, & \mbox{in}\ \Omega\\
	\bn\cdot D \nabla \phi + \dfrac{1}{2}\phi & = & \dfrac{1}{2} f(\bx), & \mbox{on}\ \partial\Omega
	\end{array}
	\end{equation}
	We therefore conclude that
	\begin{equation}\label{EQ:Diff Sigma B}
	\sigma_a^{P_2}-\sigma_a^{P_1} =\dfrac{H}{\Xi}\dfrac{\phi-(\phi_1-\frac{2}{3}\phi_2)}{\phi(\phi_1-\frac{2}{3}\phi_2)}=\dfrac{\Xi}{H\sigma_a^{P_2} \sigma_a^{P_1}}\left((\phi-\phi_1)+\frac{2}{3}\phi_2\right).
	\end{equation}
	We observe from the second equation and its boundary condition (i.e. the forth equation) in~\eqref{EQ:SP2 simp Rec A} that $\phi_2$ is the unique solution to~\eqref{EQ:SP2 simp phi2 Rec}. 

	Let $\wt\phi = \phi-\phi_1$. We verify that $\wt\phi$ solves 
	\begin{equation}\label{EQ:SP2 simp Rec B}
	\begin{array}{rcll}
	-\nabla\cdot D \nabla \wt\phi &=& 0, & \mbox{in}\ \Omega\\
	\bn\cdot D \nabla\wt \phi +\dfrac{5}{16\kappa}\wt\phi & = & 0, & \mbox{on}\ \partial\Omega
	\end{array}
	\end{equation}
	where we have used the assumption that $\sigma_a$ is known on $\partial\Omega$, made in~($\cA$-iv), so that $(\phi_1-\frac{2}{3}\phi_2)=\phi=\dfrac{H}{\Xi\sigma_a}$ on the boundary. This equation has a unique solution $\wt\phi=0$. Therefore~\eqref{EQ:Diff Sigma B} simplifies to~\eqref{EQ:Rec Sigma-a Dif}.
\end{proof}
This result says that the difference between reconstructions based on the simplified $P_2$ model~\eqref{EQ:SP2 simp} and reconstructions based on the classical diffusion model~\eqref{EQ:Diff} are noticeable. In particular, if the data are generated with the simplified diffusion model~\eqref{EQ:SP2 simp}, then using the classical diffusion model~\eqref{EQ:Diff} would give us reconstructions that are simply not as accurate, vice versa. An implication of this is that if we believe that the data we used in PAT are generated by a physical process best modeled by the radiative transport equation~\eqref{EQ:RTE}, then using the simplified $P_2$ model to perform reconstructions is advantageous to using the classical diffusion model, although we do not have an explicit characterization as in Theorem~\ref{THM:Stab Mult Model}.

\section{Numerical experiments}
\label{SEC:Num}

We now present some numerical simulations on the inverse problems studied in the previous sections. For simplicity, we consider a setup in which the physical properties and the illuminations used are invariant in the $z$ direction so that we can perform simulations in the two-dimensional case. 

The spatial variables in the simplified $P_2$ model~\eqref{EQ:SP2 simp}, the classical diffusion model~\eqref{EQ:Diff}, and the acoustic equation for ultrasound propagation~\eqref{EQ:Wave} are discretized using the finite element method with piecewise linear Lagrange elements. The time variable in the wave equation is discretized with a second-order finite difference scheme. The optical illumination sources are all selected to have strength distribution along the boundary following the Gaussian distribution with standard deviation $0.5$. The specific locations, i.e. centers, of the sources will be given later in the numerical examples.

The synthetic acoustic data we will use are generated by solving the forward diffusion models with the true optical properties and then feeding the corresponding initial pressure field $H$ into the acoustic wave equation. To mimic measurement error, we pollute the data with additional random noise by multiplying each datum point by $(1 + \sqrt{3}\eta \times 10^{-2} \texttt{random})$, with \texttt{random} a uniformly distributed random variable taking values in $[-1, 1]$, and $\eta$ being the noise level (i.e. the size of the variance in percentage). When no additional random noise are added to the synthetic data, we will say the data are ``clean'' ($\eta=0$). Otherwise, we say the data are ``noisy'' ($\eta\neq 0$).

The numerical simulations performed are all based on the minimization of the mismatch functional~\eqref{EQ:Obj} as documented in Section~\ref{SEC:Alg}. In all the simulations, the Gr\"uneisen coefficient is assumed known and $\Xi=0.5$. We emphasize that as long as $\Xi$ is assumed known, whether or not it is a constant has no visible impact on the reconstruction of $\sigma_a$ and $\sigma_s$. Moreover, under the same setup, we observe no visible differences between reconstructions in the case of constant ultrasound speed and these in the case of variable ultrasound speed. The numerical results we presented here are all obtained with $c(\bx)=1$.

We present two groups of numerical simulations. 

\subsection{Inversions based on the simplified $P_2$ model}
\label{SUBSEC:Num Acous}

In the first group, we study PAT reconstructions with the simplified $P_2$ light propagation model. That is, we generate synthetic data using the model equation system~\eqref{EQ:SP2 simp} and perform the reconstruction using the same system of equation. 

\paragraph{Experiment 1 [Reconstructing $\sigma_a$ from Acoustic Data].} In the first numerical experiment, we attempt to reconstruct the absorption coefficient assuming that the scattering coefficient, as well as the Gr\"uneisen coefficient which is set to be $\Xi=0.5$ as mentioned above, is known. More precisely, the spatial domain we take is the square $\Omega = (0, 2)\times (0, 2)$, the scattering coefficient $\sigma_s = 80$, and the anisotropic factor of the scattering kernel $g=0.9$. The true absorption coefficient has the form
\begin{equation}\label{EQ:Sigma-a-PAT1}
	\sigma_a(\bx) = 0.1 + 0.1 \chi_{B_1} + 0.2 \chi_{B_2},
\end{equation}
with $B_1 = \{(x,y) : (x-1.0)^2 + (y-1.5)^2 \le (0.2)^2 \}$ and $B_2 = \{(x,y) : (x-1.5)^2 + (y-1.0)^2 \le (0.3)^2 \}$. The acoustic data are collected from four different illuminations located respectively at the centers of the four sides of the domain.

\begin{figure}[hbt!]
	\centering
	\includegraphics[width=0.45\textwidth,height=0.2\textwidth]{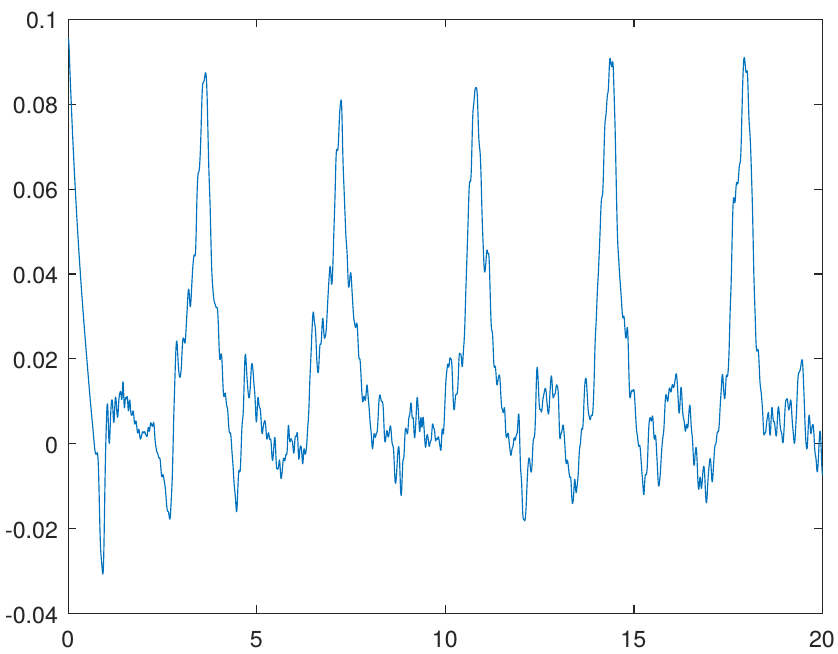}
	\includegraphics[width=0.45\textwidth,height=0.2\textwidth]{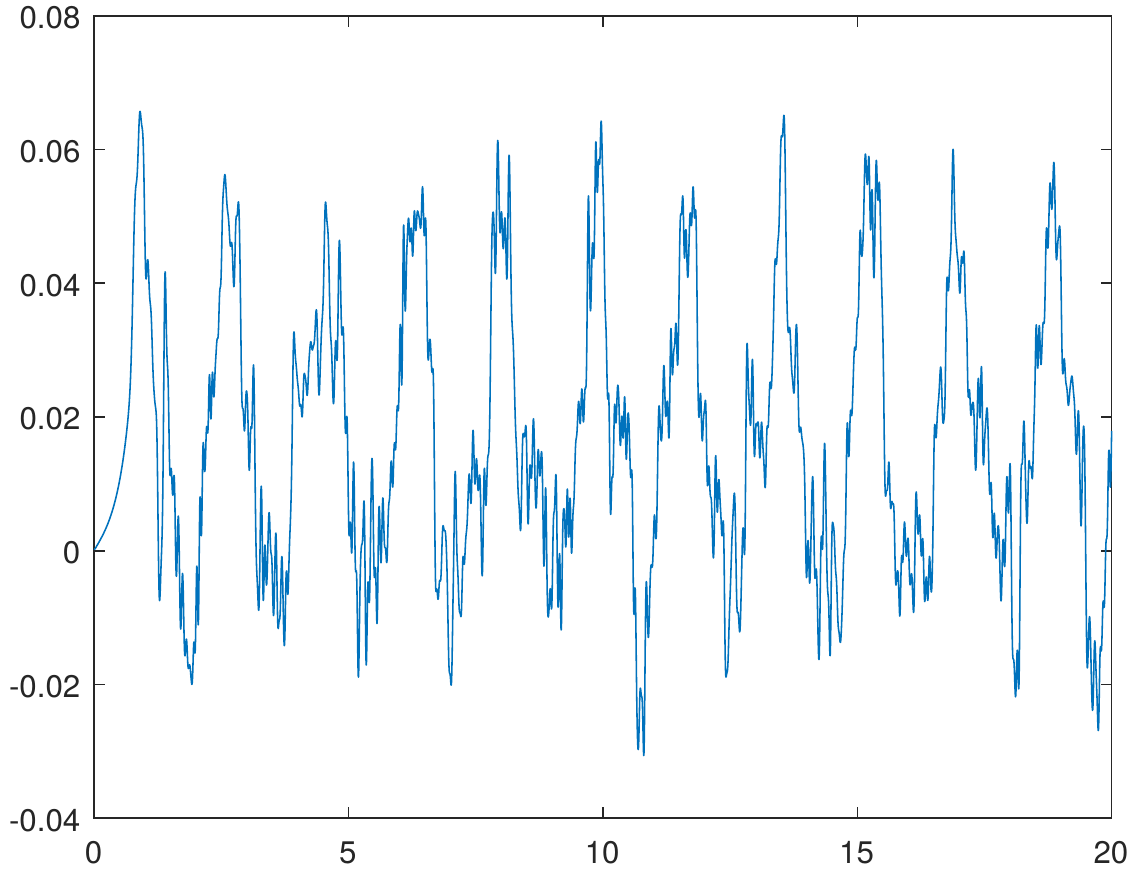}\\
	\includegraphics[width=0.45\textwidth,height=0.2\textwidth]{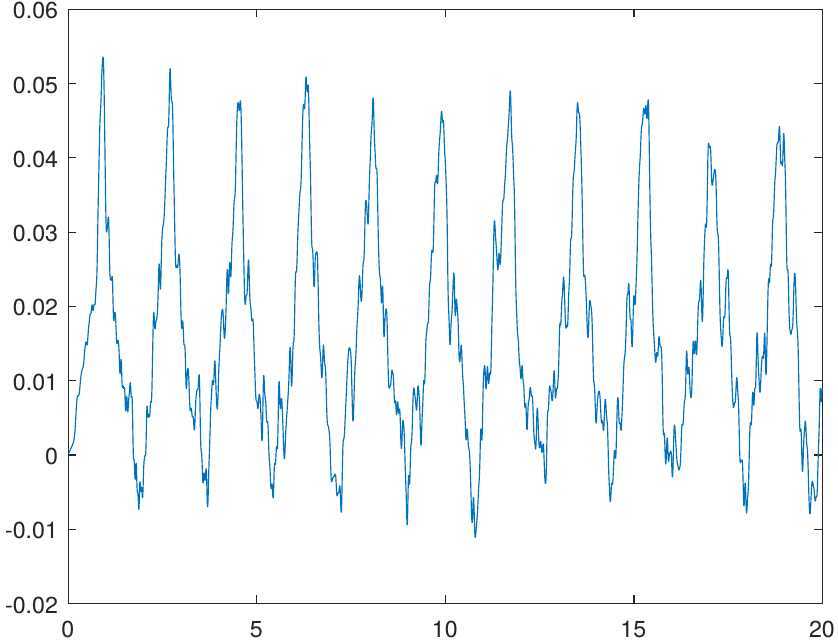}
	\includegraphics[width=0.45\textwidth,height=0.2\textwidth]{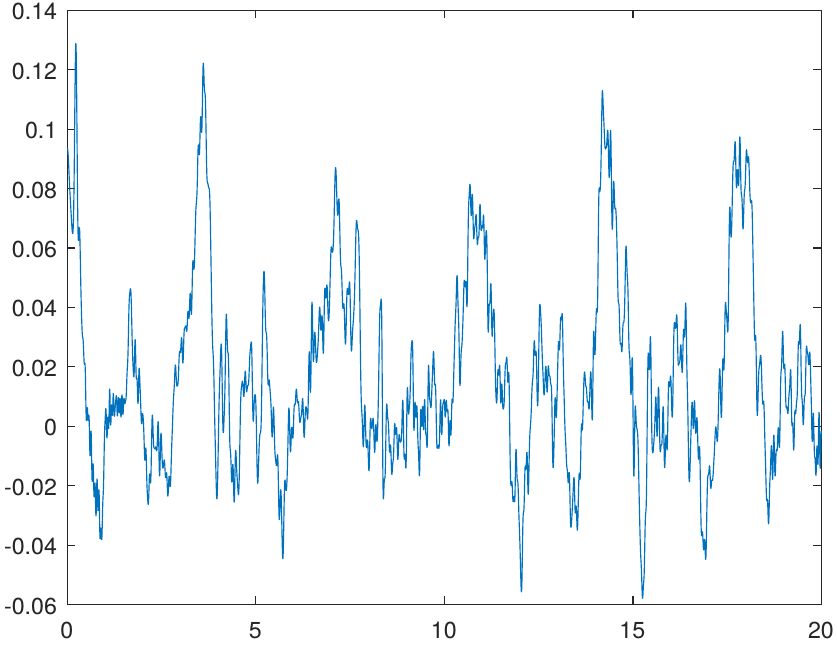}
	\caption{Ultrasound signals measured at two different locations ($(0.0, 0.8)$ (left) and $(1.0, 1.2)$ (right)) for two different illumination sources (top row: source located on the left boundary; bottom row: source located on the right boundary) in the time window $(0, 20)$.}
	\label{FIG:Ultrasound Signals}
\end{figure}
We first show in Figure~\ref{FIG:Ultrasound Signals} some typical acoustic signals we recorded in this setup. Shown are signals measured at $(0.0, 0.8)$ and $(1.0, 1.2)$ respectively for two different optical illuminations in the time interval $(0, T=20)$. Note that the data used in the reconstructions in the rest of the paper are on a larger time interval with $T=40$, even though we observe in our numerical experiments that $T=10$ is often more than enough for stable reconstructions.

In Figure~\ref{FIG:Sigma-a-PAT1}, we show the reconstruction of the absorption coefficient ~\eqref{EQ:Sigma-a-PAT1} using data with noise levels $\eta=0$ (i.e. noise-free data) and $\eta=5$ respectively. The algorithm parameters are as follows. The initial guess for all the reconstructions is $\sigma_a^0=0.1$. The linear bounds we impose on the absorption coefficient are very loose: $10^{-3} \leq \sigma_a \leq 0.5$. The regularization parameter in~\eqref{EQ:Obj} was chosen as $\alpha = 1e-9$ by trial and error.
\begin{figure}[hbt!]
	\centering
	\includegraphics[width=0.3\textwidth]{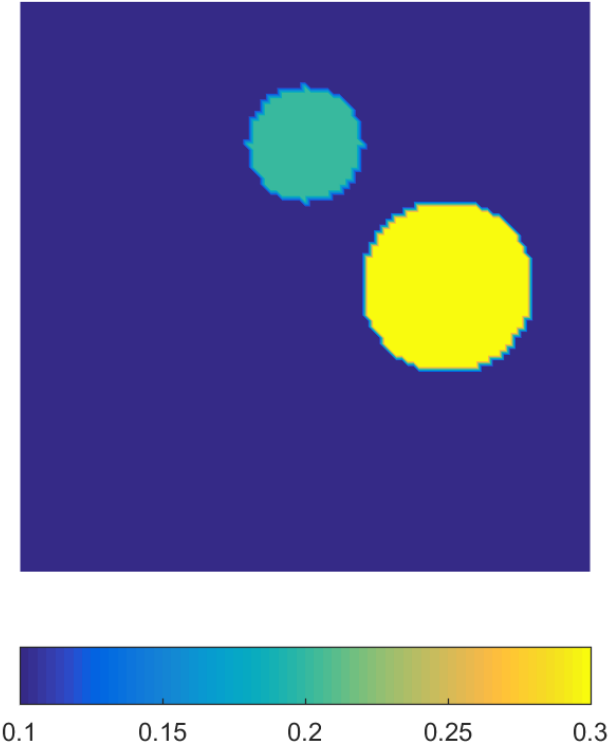}
	\includegraphics[width=0.3\textwidth]{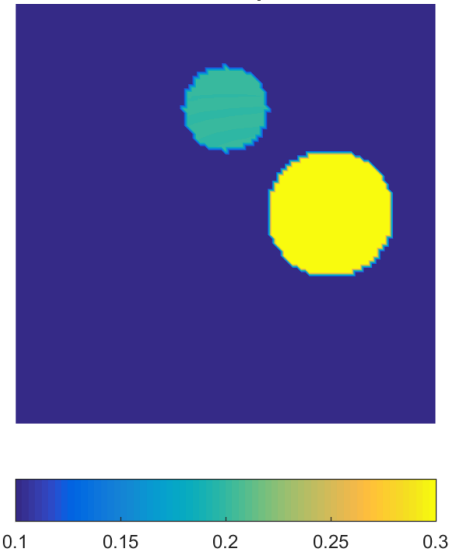}\hskip 0.1cm
	\includegraphics[width=0.3\textwidth]{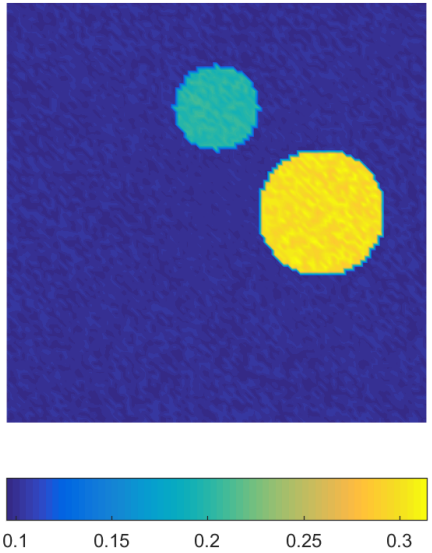}
	\caption{The absorption coefficient $\sigma_a$ in~\eqref{EQ:Sigma-a-PAT1} (left) and the reconstructions using noise-free data ($\eta=0$, middle) and noisy data ($\eta=5$, right).}
	\label{FIG:Sigma-a-PAT1}
\end{figure}

Visual observation confirms that the reconstructions are of very high quality in this case, comparable to the quality of reconstructions for PAT using the radiative transfer model~\cite{HaNeNgRa-SIAM18,MaRe-CMS14,SaTaCoAr-IP13} and the classic diffusion model~\cite{BaRe-IP11,CoArKoBe-AO06}. To quantitatively measure the quality of the reconstructions, we compute the relative $L^2$ distance between the reconstructions and the true coefficient. This distance is defined as 
\[
	\cE=\dfrac{\|\hat \sigma_a-\sigma_a\|_{L^2(\Omega)}}{\|\sigma_a\|_{L^2(\Omega)}},
\]
where $\sigma_a$ and $\hat\sigma_a$ are respectively the true and reconstructed absorption coefficients. The relative $L^2$ error for the reconstructions in Figure~\ref{FIG:Sigma-a-PAT1} are respectively $\cE=0.006$ and $\cE=0.035$ for the case of $\eta=0$ and $\eta=5$. The reconstruction results are very stable with respect to the initial guess we used, and the linear bounds we imposed on $\sigma_a$ do not play a major role in this case either. These observations indicate that the inverse problem is fairly well-conditioned.


\begin{figure}[hbt!]
	\centering
	\includegraphics[width=0.3\textwidth]{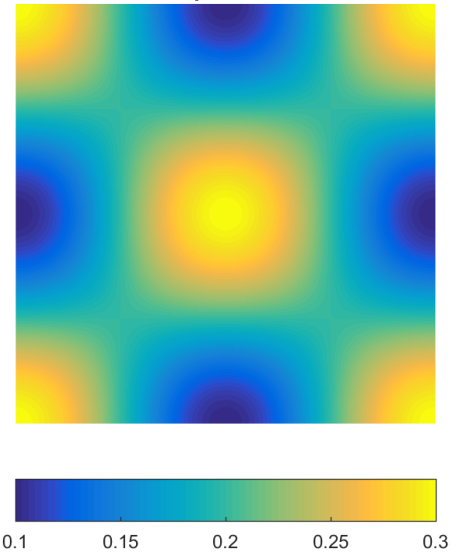}
	\includegraphics[width=0.3\textwidth]{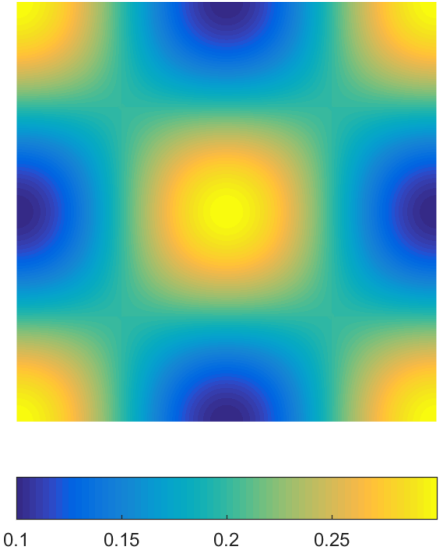}\hskip 0.2cm
	\includegraphics[width=0.3\textwidth]{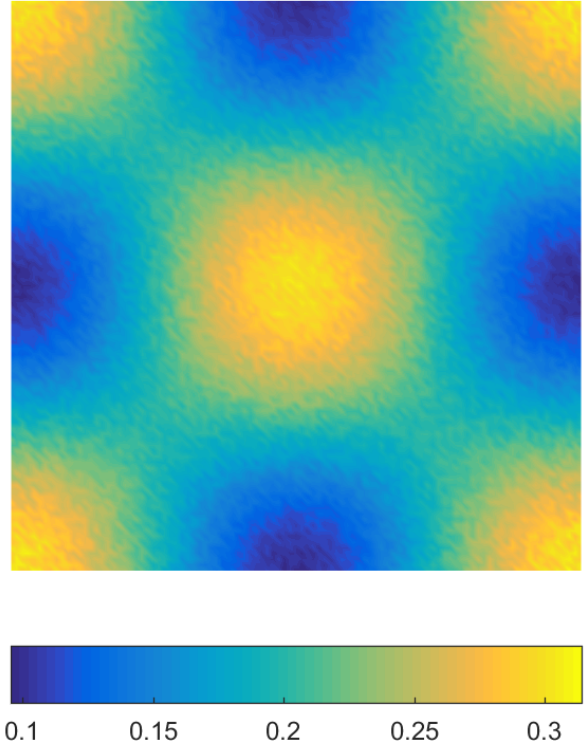}
	\caption{The absorption coefficient $\sigma_a$ in~\eqref{EQ:Sigma-a-PAT2} (left) and the reconstructions using noise-free data ($\eta=0$, middle) and noisy data ($\eta=5$, right).}
	\label{FIG:Sigma-a-PAT2}
\end{figure}
We repeat similar numerical experiments for a few other absorption coefficients. The quality of the reconstruction results are very similar to the case we reported above. For instance, in Figure~\ref{FIG:Sigma-a-PAT2}, we show the reconstructions of a smooth absorption map defined as
\begin{equation}\label{EQ:Sigma-a-PAT2}
\sigma_a (\bx) = 0.2 + 0.1\cos(\pi x-\pi)\cos(\pi y-\pi).
\end{equation}
We again used the four sources to generate four data sets and initialize the reconstruction algorithm at $\sigma_a^0 = 0.1$, very different from the true coefficient. We also impose the same pointwise inequality constraints on the absorption coefficient, i.e. $10^{-3} \leq \sigma_a \leq 0.5$, and the same regularization parameter $\alpha$. The reconstruction quality is comparable to the smooth case and with results from the diffusion model~\cite{CoArKoBe-AO06}. The reconstruction errors are concentrated at the discontinuities, as expected. The relative $L^2$ error in the reconstructions are respectively $\cE=0.004$ and $\cE=0.034$ for the cases of $\eta=0$ and $\eta=5$.

\begin{figure}[hbt!]
	\centering
	\includegraphics[width=0.3\textwidth]{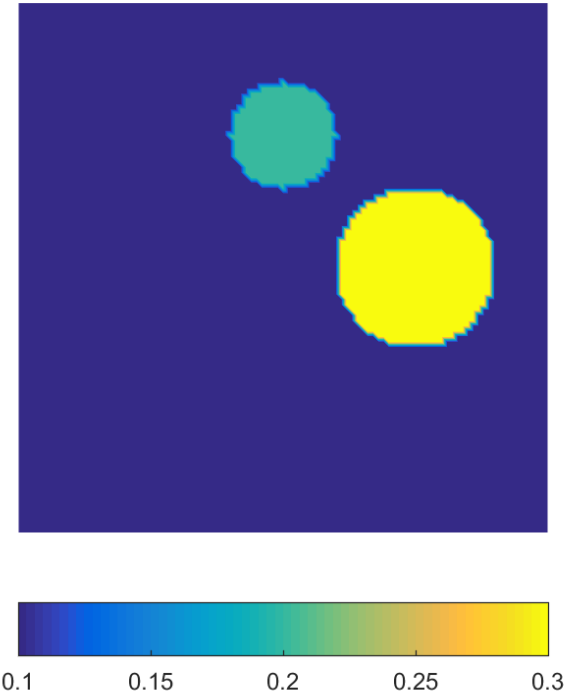}
	\includegraphics[width=0.3\textwidth]{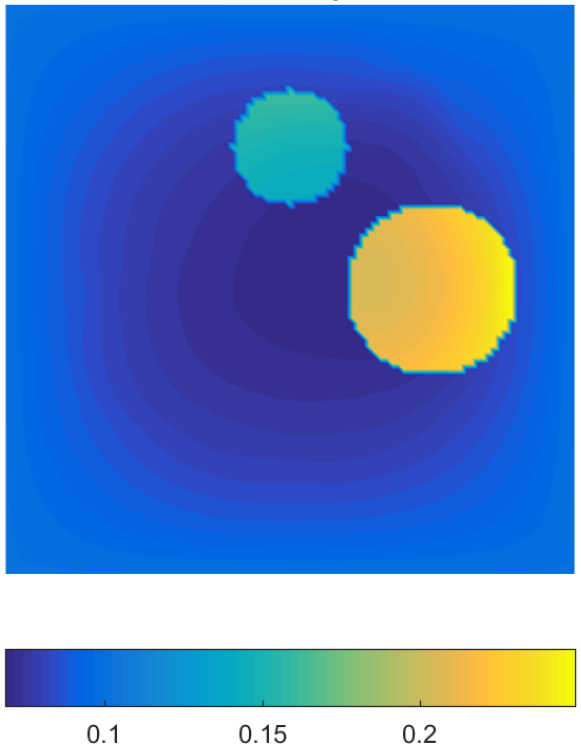}
	\includegraphics[width=0.3\textwidth]{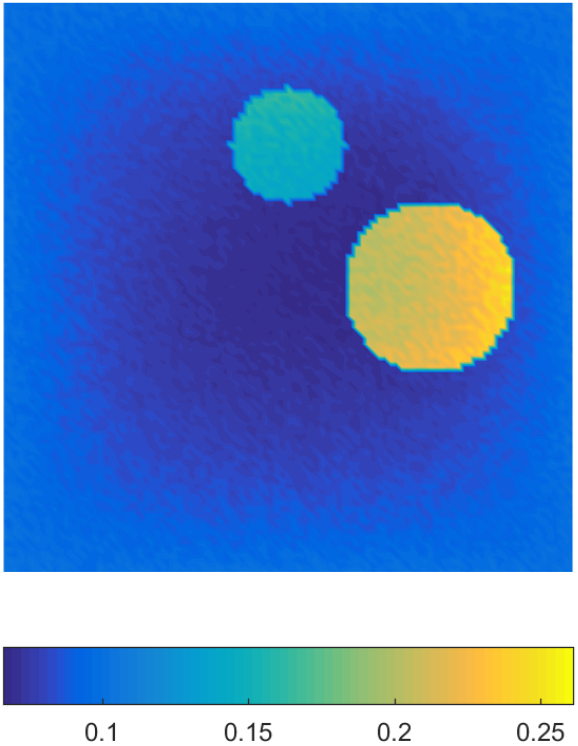}\\
	\includegraphics[width=0.3\textwidth]{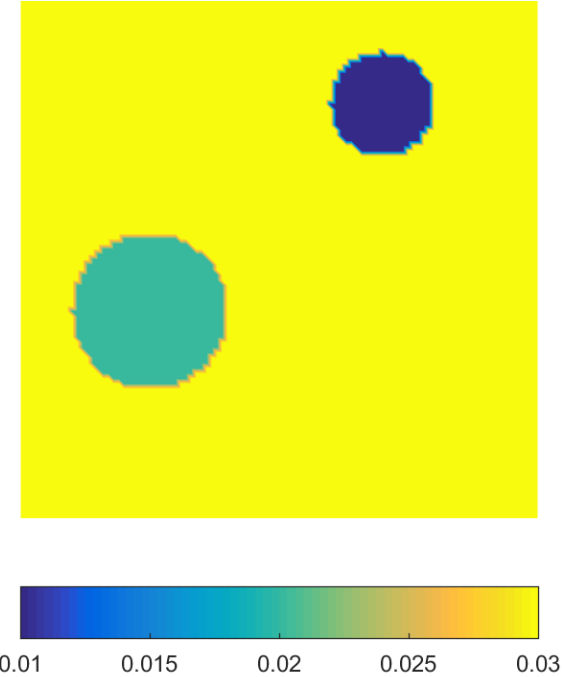}
	\includegraphics[width=0.3\textwidth]{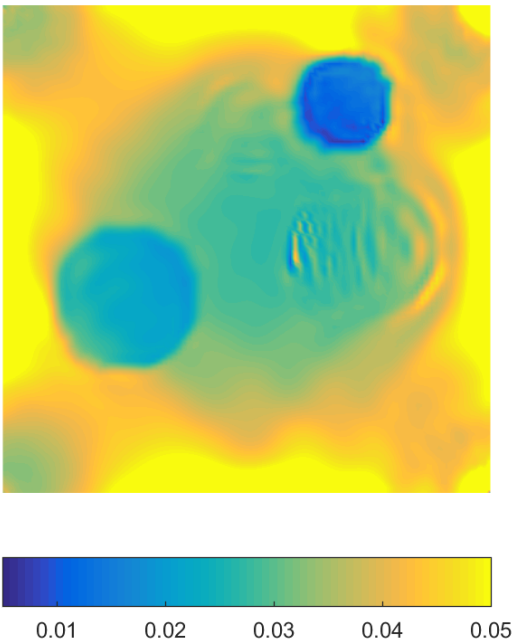}
	\includegraphics[width=0.3\textwidth]{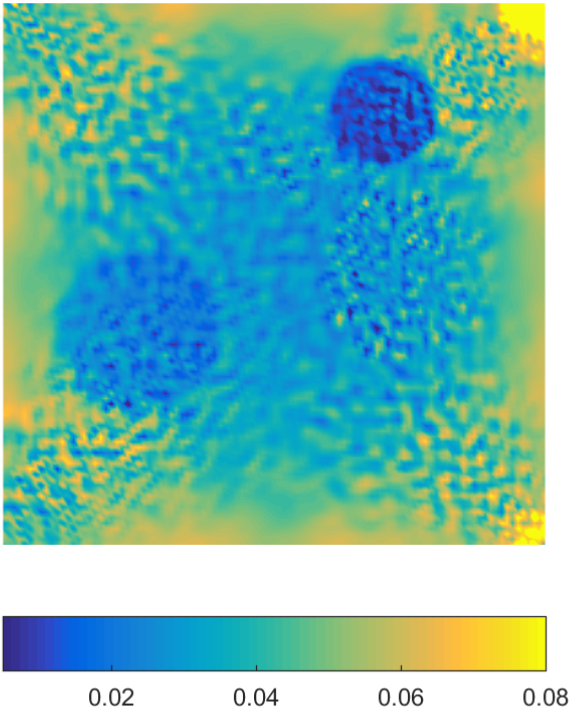}
	\caption{Reconstructions of the absorption coefficient $\sigma_a$ (defined in~\eqref{EQ:Sigma-a-PAT1}, top row) and the diffusion coefficient ($D=1/[3(1-g)\sigma_s]$ with $\sigma_s$ defined in~\eqref{EQ:Sigma-a-s-PAT-s}, bottom row). Show are the true coefficients (left) and the reconstructions using noise-free data (middle) and noisy data with $\eta = 5$ (right).}
	\label{FIG:Sigma-a-s-PAT}
\end{figure}
\paragraph{Experiment 2 [Reconstructing $(\sigma_a, \sigma_s)$ from Acoustic Data].} In this numerical experiment, we perform simultaneous reconstruction of the absorption and the scattering coefficients. The absorption coefficient is the same as the one defined in~\eqref{EQ:Sigma-a-PAT1} while the scattering coefficient is defined as
\begin{equation}\label{EQ:Sigma-a-s-PAT-s}
	\sigma_s(\bx) = 85 + 260\chi_{B_2}(\bx) + 260\chi_{B_3}(\bx), 
\end{equation}
with $B_3 = \{(x,y) : (x-0.5)^2 + (y-0.8)^2 \le (0.3)^2 \}$ and $B_4 = \{(x,y) : (x-1.4)^2 + (y-1.6)^2 \le (0.2)^2 \}$. We again collect data from four different illumination patterns. The reconstruction are initialized at $\sigma_a^0 = 0.1$ and $\sigma_s^0 = 120$. The linear bound constraints are set as $10 \leq \sigma_s \leq 500$ and $0 \leq \sigma_a \leq 1$ in all cases. The regularization strength are selected at $\alpha = 1e-6$ and $\beta = 1e-10$ after a couple of trial and error testings. Note that the discrepancy between the parameter $\alpha$ and $\beta$ mainly come from the fact that the coefficients $\sigma_a$ and $\sigma_s$ have values that are different on a few orders of magnitude: $\sigma_a\sim 0.1$ while $\sigma_s \sim 100$.

\begin{table}[hbt!]
	\centering
	\begin{tabular}{ccccc}
	\hline
	Noise Level & & $\cE_{a}=\frac{\|\hat\sigma_a-\sigma_a\|_{L^2}}{\|\sigma_a\|_{L^2}}$ & & $\cE_{s}=\frac{\|\hat\sigma_s-\sigma_s\|_{L^2}}{\|\sigma_s\|_{L^2}}$\\
	\hline
	$\eta=0$ && 0.023 && 0.182\\
	$\eta=2$ && 0.043 && 0.251\\
	$\eta=5$ && 0.091 && 0.284\\
	$\eta=10$ && 0.174 && 0.385\\
	\hline
	\end{tabular}
	\caption{Relative $L^2$ errors in the simultaneous reconstructions of the absorption and scattering coefficients given in~\eqref{EQ:Sigma-a-PAT1} and ~\eqref{EQ:Sigma-a-s-PAT-s} from acoustic data with different noise levels.}
	\label{TAB:Error}
\end{table}
In Figure~\ref{FIG:Sigma-a-s-PAT} we show the reconstructions results from clean data and noisy data with $\eta=5$. The relative $L^2$ errors in these reconstructions, as well as two additional reconstructions, are summarized in Table~\ref{TAB:Error}. 

We observe that the quality of the reconstructions is again very high when noise level is very low. However, the quality degenerates quickly as the noise level increases, especially for the scattering coefficient. As we observed in the previous cases, the box constraints on the coefficients are very loose and do not have significant impact on the reconstructions. Tighter bounds can be imposed to improve the quality of the reconstructions when these \emph{a priori} information are available. More carefully selection of the regularization coefficient might also help improve the reconstructions. Those are not the directions that we want to pursue in this research.

\subsection{Inversions for cross-model comparisons}
\label{SUBSEC:Num Cross}

The numerical tests in Experiment 1 and Experiment 2 suggest that the PAT inverse problem based on the simplified $P_2$ model has very similar stability properties and reconstruction quality as those based on the classical diffusion model or the radiative transport model~\cite{BaRe-IP11,BaRe-IP12,MaRe-CMS14,SaTaCoAr-IP13}, assuming that the data used in the reconstructions are generated from the same model. In the next numerical experiment, we address a different issue in PAT reconstructions. We are interested in studying the impact of model inaccuracies on the quality of the reconstructions of the absorption and the scattering coefficients. More precisely, assuming that the measurement data are generated by an accuracy model but we perform reconstructions based on a less accurate model. We are interested in the impact of the inaccuracy of the inversion model on the reconstruction results. This problem has been studied in Theorem~\ref{THM:Stab Mult Model} for the simple case of reconstructing only the absorption coefficient. We now provide some numerical evidences.
\begin{figure}[hbt!]
	\centering
	\includegraphics[width=0.35\textwidth]{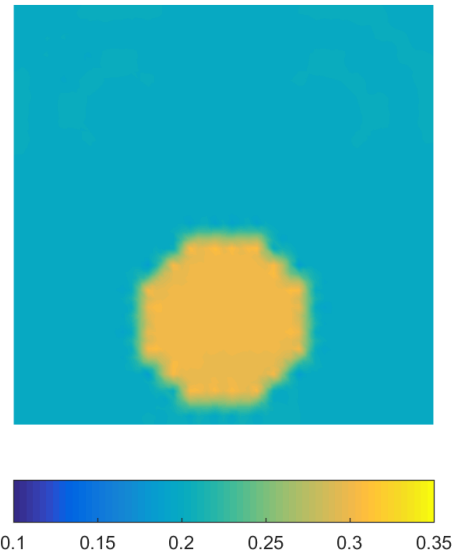}\hskip 0.6cm
	\includegraphics[width=0.355\textwidth]{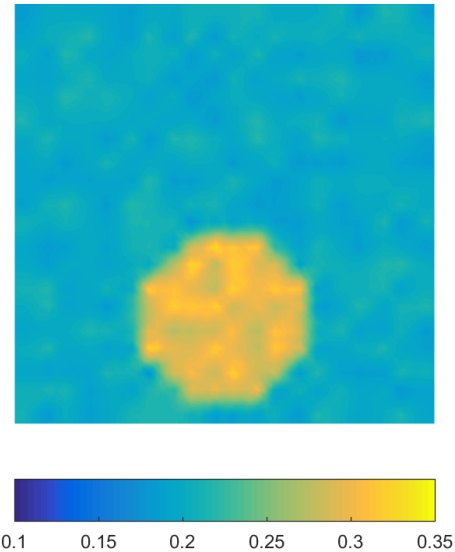}\\
	\includegraphics[width=0.35\textwidth]{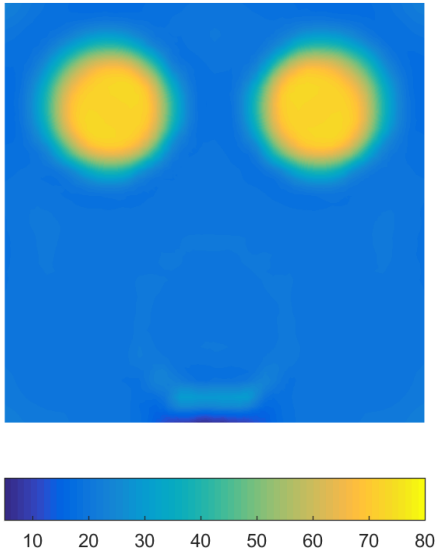}\hskip 0.6cm
	\includegraphics[width=0.35\textwidth]{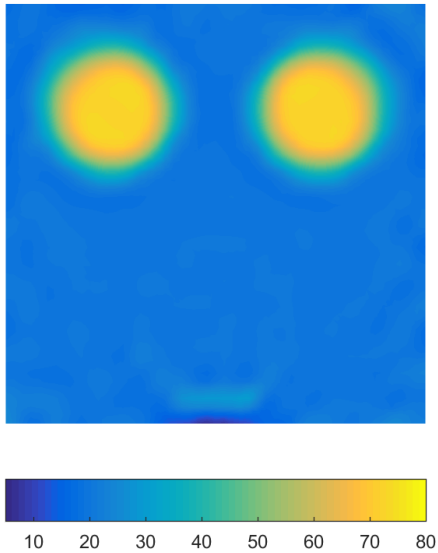} 
	\caption{Reconstructions of the absorption coefficient $\sigma_a$ (top row) and the scattering coefficient $\sigma_s$ (bottom row) defined in~\eqref{EQ:Sigma-a-s-QPAT} from noise-free data ($\eta=0$, left column) and noisy data with $\eta = 5$ (right column). The reconstructions are performed using the simplified $P_2$ model as the forward light propagation model.}
	\label{FIG:Sigma-a-s-QPAT-P2P2}
\end{figure}

\paragraph{Experiment 3 [Cross-model Inversion Comparisons].}
The setup is as follows. The domain is the unit square $\Omega=(0,1)\times(0,1)$. The true absorption and scattering coefficients are
\begin{equation}\label{EQ:Sigma-a-s-QPAT}
\begin{array}{rcl}
\sigma_a (\bx) &=& 0.2 + 0.1\chi_{B_5}(\bx), \\
\sigma_s (\bx) &=& 20 + 60\chi_{B_6}(\bx) + 60\chi_{B_7}(\bx), 
\end{array}
\end{equation}
with $B_5 = \{(x,y) : (x-0.5)^2 + (y-0.25)^2 \le (0.2)^2 \}$, $B_6 = \{(x,y) : (x-0.25)^2 + (y-0.75)^2 \le (0.15)^2 \}$, and $B_7 = \{(x,y) : (x-0.75)^2 + (y-0.75)^2 \le (0.15)^2 \}$. The anisotropic factor is again $g=0.9$. 

We collect data from the four different illumination patterns that we used in the previous numerical experiments. The data are generated by solving the simplified $P_2$ model with the true absorption and scattering coefficients. As a benchmark, we solved the radiative transport model in this case with the finite volume discrete ordinate scheme developed in~\cite{ReBaHi-SIAM06}. We observe that the simplified $P_2$ internal data $H^{P_2}$ is much closer than the classical diffusion internal data $H^{diff}$ to the radiative transport internal data $H^{rte}$ with the same optical coefficients.

To exclude the impact of the acoustic wave model on the comparison, we perform reconstructions directly from the internal data $H$, not from the boundary ultrasound data as in the previous numerical experiments. In all the reconstruction results below, we initialize the inversion algorithm at $\sigma_a^0 = 0.1$ and $\sigma_s^0 = 50$. The linear bound constraints are set as $10 \leq \sigma_s \leq 100$ and $0 \leq \sigma_a \leq 1$ in all cases. The regularization strength are selected at $\alpha = 1e-3$ and $\beta = 1e-8$ after extensive numerical tests.

\begin{table}[hbt!]
	\centering
	\begin{tabular}{ccccc}
	\hline
	Noise Level & & $\cE_a=\frac{\|\hat\sigma_a-\sigma_a\|_{L^2}}{\|\sigma_a\|_{L^2}}$ & & $\cE_s=\frac{\|\hat\sigma_s-\sigma_s\|_{L^2}}{\|\sigma_s\|_{L^2}}$\\
	\hline
	$\eta=0$ && 0.039 && 0.244\\
	$\eta=2$ && 0.041 && 0.244\\
	$\eta=5$ && 0.052 && 0.245\\
	$\eta=10$ && 0.079 && 0.250\\
	\hline
	\end{tabular}
	\caption{Relative $L^2$ errors in the simultaneous reconstructions of the absorption and scattering coefficients in~\eqref{EQ:Sigma-a-s-QPAT} based on the simplified $P_2$ model data with different noise levels.}
	\label{TAB:P2P2-Error}
\end{table}
In Figure~\ref{FIG:Sigma-a-s-QPAT-P2P2} and Figure~\ref{FIG:Sigma-a-s-QPAT-P2P1} we show respectively the reconstruction results using the simplified $P_2$ model and the classical $P_1$ model as the model for light propagation. The relative errors in the reconstructions for the two groups of numerical simulations are summarized in Table~\ref{TAB:P2P2-Error} and Table~\ref{TAB:P2P1-Error} respectively.

\begin{figure}[hbt!]
	\centering
	\includegraphics[width=0.35\textwidth]{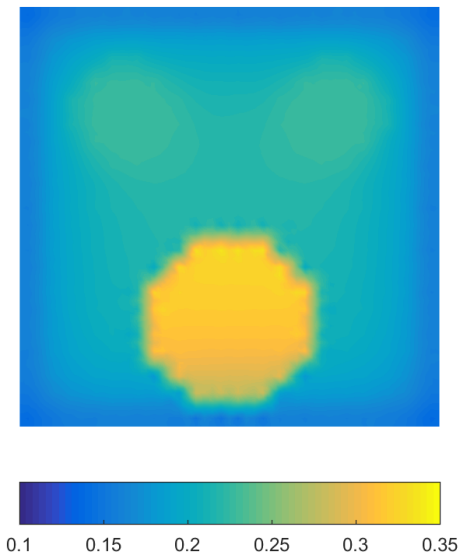} \hskip 1em
	\includegraphics[width=0.345\textwidth]{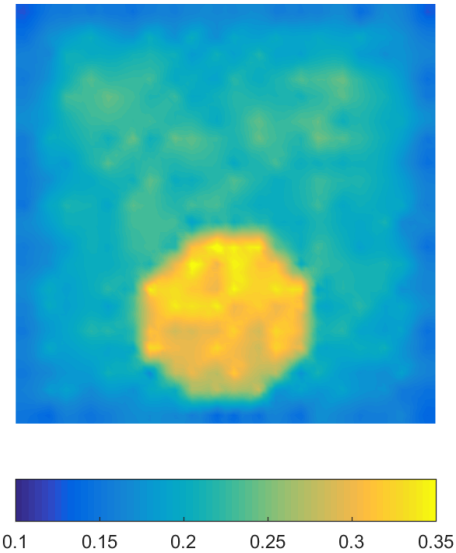}\\
	\includegraphics[width=0.35\textwidth]{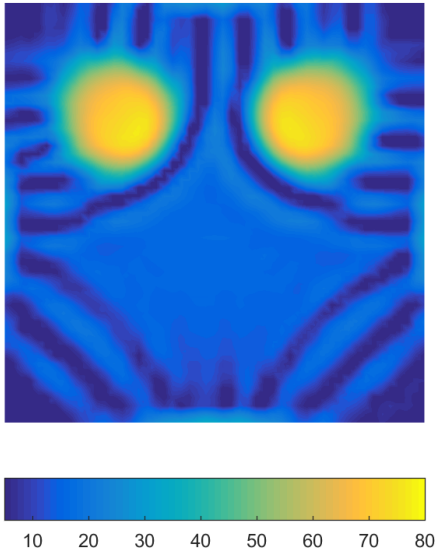} \hskip 1.7em
	\includegraphics[width=0.35\textwidth]{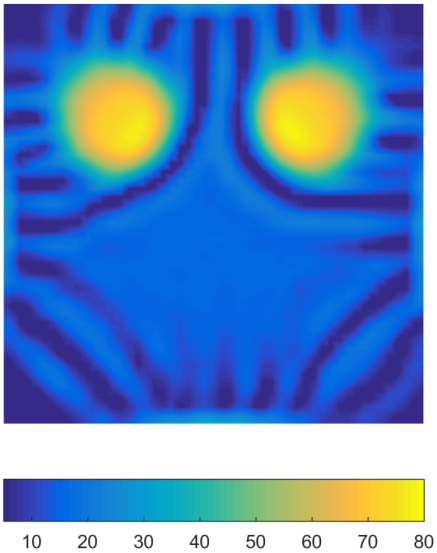}\\
	\caption{Same as Figure~\ref{FIG:Sigma-a-s-QPAT-P2P2} except that the reconstructions are performed using the $P_1$ model~\eqref{EQ:Diff} as the forward light propagation model.}
	\label{FIG:Sigma-a-s-QPAT-P2P1}
\end{figure}
A quick comparison between the first row of Figure~\ref{FIG:Sigma-a-s-QPAT-P2P2} and that of Figure~\ref{FIG:Sigma-a-s-QPAT-P2P1} shows that the reconstructions are significantly different, even though the internal data used in the reconstructions are the same. To be more precise, the relative $L^2$ errors in the reconstructions changed from roughly $(0.04, 0.24)$ in the first row of Figure~\ref{FIG:Sigma-a-s-QPAT-P2P2} to roughly $(0.12, 0.39)$ in the first row of Figure~\ref{FIG:Sigma-a-s-QPAT-P2P1}. This shows that the right hand side of equation~\eqref{EQ:Rec Sigma-a Dif} is relatively large, which indicates that in this specific setting, the solutions to the simplified $P_2$ model and the classical diffusion model are quite different.

\begin{table}[hbt!]
	\centering
	\begin{tabular}{ccccc}
	\hline
	Noise Level & & $\cE_a=\frac{\|\hat\sigma_a-\sigma_a\|_{L^2}}{\|\sigma_a\|_{L^2}}$ & & $\cE_s=\frac{\|\hat\sigma_s-\sigma_s\|_{L^2}}{\|\sigma_s\|_{L^2}}$\\
	\hline
	$\eta=0$ && 0.115 && 0.385\\
	$\eta=2$ && 0.116 && 0.388\\
	$\eta=5$ && 0.120 && 0.383\\
        $\eta=10$&& 0.134 && 0.387\\
	\hline
	\end{tabular}
	\caption{Same as Table~\ref{TAB:P2P2-Error} except that the reconstructions are performed using the $P_1$ model~\eqref{EQ:Diff} as the forward light propagation model.}
	\label{TAB:P2P1-Error}
\end{table}
The last row of Table~\ref{TAB:P2P2-Error} shows the reconstruction results with data containing $10\%$ random noise. Comparing the results with these in the first row of Table~\ref{TAB:P2P1-Error} shows that the former is still better. This implies in some sense that the ``noise'' we introduced here, by using the classical diffusion model~\eqref{EQ:Diff} to replace the simplified $P_2$ model~\eqref{EQ:SP2 simp}, is larger than $10\%$. Therefore, if we believe that the data are generated with accurate model, using the same model to do PAT reconstructions gives better results than using a less accurate model. This is in general true for most of the inverse problems we know. However, for problems such as optical tomography, the benefit of using more accurate models in reconstructions is lost at relatively low noise level~\cite{ChDe-OE09,ReBaHi-AO07,WrScAr-MST07}, which is mainly due to the severe ill-conditioning of the inversion problem. In our case, the inversion is less ill-conditioned (roughly, not mathematically speaking, although this can be characterized mathematically), and the accuracy of the forward model plays a more important role.

\section{Concluding remarks}
\label{SEC:Concl}

We studied in the paper the problem of reconstructing optical absorption and scattering coefficients in quantitative photoacoustic tomography with the simplified $P_2$ model as the model of light propagation in the underlying medium. We showed numerically that one can reconstruct the absorption and scattering coefficients from ultrasound data generated under multiple illuminations, in a relatively stable manner. We also studied the quantitative step of the reconstructions where we developed some uniqueness and stability results under simplified circumstances.

Let us emphasize that we are not claiming that using the simplified $P_2$ model will always be better than using the radiative transfer model or the classical diffusion model. There are many factors that we have to consider when deciding which model to use in a specific situation. Let us assume that we have a working inversion scheme that can reconstruct the internal data $H$ from given ultrasound measurements, with either constant or variable ultrasound speed. Then in terms of accuracy, the simplified $P_2$ model (or even the radiative transport model) should be used when the difference between data $H$ predicted by the radiative transport model and these predicted by the simplified $P_2$ model dominate the noise in the measured internal data $H^*$. Otherwise, the classical diffusion model is sufficient and should be used. In terms of computational cost, the simplified $P_2$ model is comparable to the classical diffusion model, and they are both significantly cheaper than the radiative transport model.

There are multiple aspects of the current research that can be improved. In terms of the comments in the previous paragraph, an important issue to address is to perform similar reconstructions from experimentally measured ultrasound data. It would be interesting to see whether or not we can observe any difference between the reconstructions with the simplified $P_2$ model and those with the classical diffusion model with experimental data.  It is also of great interests to generalize method proposed in this work to the case of the multiple wavelength data. In that setup, we hope to be able to simultaneously reconstruct the absorption, and the scattering and the Gr\"uneisen coefficients as proved in the classical diffusion case~\cite{BaRe-IP11}.  
\section*{Acknowledgments}

This work is partially supported by the National Science Foundation through grants DMS-1321018, DMS-1620473 and DMS-1720306. S.V. would like to acknowledge partial support from the Statistical and Applied Mathematical Sciences Institute (SAMSI).


{\small

}


\end{document}